\documentclass{amsart}
\title[The Canonical Quadratic Pair on Clifford Algebras over Schemes]{The Canonical Quadratic Pair on \\ Clifford Algebras over Schemes}

\usepackage{amsmath, amssymb, amsthm}
\usepackage{bbm, lmodern, mathrsfs, euscript}
\usepackage{tikz, tikz-cd}
\usepackage{xcolor}
\usepackage{arydshln}
\usepackage{enumitem}
\usepackage[colorlinks=true, pdfstartview=FitV, linkcolor=blue, citecolor=blue, urlcolor=blue, breaklinks=true]{hyperref}
\usepackage{cleveref}

\usepackage{forloop}
\newcounter{fonts}
\let\eeee\edef
\forloop{fonts}{1}{\value{fonts} < 27}{%
\expandafter\eeee\csname \Alph{fonts}\Alph{fonts}\endcsname{\noexpand\mathbb{\Alph{fonts}}}
\expandafter\eeee\csname b\Alph{fonts}\endcsname{\noexpand\mathbf{\Alph{fonts}}}
\expandafter\eeee\csname c\Alph{fonts}\endcsname{\noexpand\mathcal{\Alph{fonts}}}
\expandafter\eeee\csname f\Alph{fonts}\endcsname{\noexpand\mathfrak{\Alph{fonts}}}
\expandafter\eeee\csname s\Alph{fonts}\endcsname{\noexpand\mathscr{\Alph{fonts}}}
\expandafter\eeee\csname e\Alph{fonts}\endcsname{\noexpand\EuScript{\Alph{fonts}}}
}
\mathchardef\mdash="2D

\newtheorem{thm}{Theorem}[section]
\newtheorem{defn}[thm]{Definition}
\newtheorem{prop}[thm]{Proposition}
\newtheorem{cor}[thm]{Corollary}
\newtheorem{lem}[thm]{Lemma}

\newtheoremstyle{remarkstyle}{\topsep}{\topsep}{\rm}{}{\bfseries}{.}{.5em}{}
\theoremstyle{remarkstyle}
\newtheorem{rem}[thm]{Remark}

\newtheorem{ex}[thm]{Example}

\newtheoremstyle{solutionstyle}{\topsep}{\topsep}{\rm}{}{\it}{.}{.5em}{}
\theoremstyle{solutionstyle}

\newtheorem*{sol*}{Solution}

\newcommand{\und}{\underline{\hspace{2ex}}}
\newcommand{\iso}{\xrightarrow{\sim}}
\newcommand{\inj}{\hookrightarrow}
\newcommand{\surj}{\twoheadrightarrow}

\newcommand{\runity}{\boldsymbol{\mu}}
\newcommand{\GL}{\mathbf{GL}}

\newcommand{\PGL}{\mathbf{PGL}}

\newcommand{\PGO}{\mathbf{PGO}}

\newcommand{\slLie}{\mathfrak{sl}}

\DeclareMathOperator{\Grp}{\mathfrak{Grp}}
\DeclareMathOperator{\Ab}{\mathfrak{Ab}}
\DeclareMathOperator{\Sets}{\mathfrak{Sets}}
\DeclareMathOperator{\Rings}{\mathfrak{Rings}}

\DeclareMathOperator{\Sch}{\mathfrak{Sch}}
\DeclareMathOperator{\Aff}{\mathfrak{Aff}}

\newcommand{\cHom}{\mathcal{H}\hspace{-0.4ex}\textit{o\hspace{-0.2ex}m}}
\newcommand{\cEnd}{\mathcal{E}\hspace{-0.4ex}\textit{n\hspace{-0.2ex}d}}
\newcommand{\cAut}{\mathbf{Aut}}
\newcommand{\bAut}{\mathbf{Aut}}
\newcommand{\cSym}{\mathcal{S}\hspace{-0.5ex}\textit{y\hspace{-0.3ex}m}}
\newcommand{\cSymd}{\mathcal{S}\hspace{-0.5ex}\textit{y\hspace{-0.3ex}m\hspace{-0.2ex}d}}
\newcommand{\cSkew}{\mathcal{S}\hspace{-0.4ex}\textit{k\hspace{-0.2ex}e\hspace{-0.3ex}w}}
\newcommand{\cAlt}{\mathcal{A}\hspace{-0.1ex}\textit{$\ell$\hspace{-0.3ex}t}}
\newcommand{\cIsom}{\mathcal{I}\hspace{-0.35ex}\textit{s\hspace{-0.15ex}o\hspace{-0.15ex}m}}

\DeclareMathOperator{\Img}{Img}
\DeclareMathOperator{\Ker}{Ker}

\DeclareMathOperator{\Hom}{Hom}

\DeclareMathOperator{\Isom}{Isom}
\DeclareMathOperator{\Aut}{Aut}

\DeclareMathOperator{\Id}{Id}

\DeclareMathOperator{\Trd}{Trd}

\DeclareMathOperator{\Char}{char}
\DeclareMathOperator{\Spec}{Spec}
\DeclareMathOperator{\Inn}{Inn}
\DeclareMathOperator{\Sym}{Sym}

\DeclareMathOperator{\Mat}{M}

\newcommand{\op}{^\mathrm{op}}

\newcommand{\Cl}{\mathcal{C}\hspace{-0.3ex}\ell}

\begin{document}
\author[C. Ruether]{Cameron Ruether}
\address{Department of Mathematics and Statistics, Memorial University of Newfoundland, St. John's, NL, Canada, A1C 5S7}
\email{cameronruether@gmail.ca}

\thanks{This work was supported by the NSERC grants of Mikhail Kotchetov (Discovery Grant 2018-04883) and Yorck Sommerh\"auser (RGPIN-2017-06543), as well as by the AARMS collaborative research group ``Groups, Rings, Lie and Hopf Algebras" and the Atlantic Algebra Centre.}

\date{September 6, 2023.}

\maketitle

\noindent{\bf Abstract:} {Working over an arbitrary base scheme $S$, we define the canonical quadratic pair on the Clifford algebra associated to an Azumaya algebra with quadratic pair. Given an Azumaya algebra $\cA$ with quadratic pair, i.e., with an orthogonal involution and a semi-trace, its associated Clifford algebra's canonical involution is only orthogonal in certain cases, namely when $\deg(\cA)$ is divisible by $8$ or when both $2=0$ over $S$ and $\deg(\cA)$ is divisible by $4$. When $\deg(\cA) \geq 8$, our definition of the canonical quadratic pair on the Clifford algebra is extended from previous work of Dolphin and Qu\'eguiner-Mathieu, who worked over fields of characteristic $2$. When $\deg(\cA)=4$, we show that no canonical quadratic pair exists.}
\medskip

\bigskip

\section*{Introduction}
{%
\renewcommand{\thethm}{\Alph{thm}}
In The Book of Involutions (\cite{KMRT}), the final chapters are dedicated to the phenomenon of \emph{triality} which arises from the order three symmetries of the Dynkin diagram $D_4$. However, these final chapters assume that the base field is of characteristic different from $2$, an assumption which is avoided throughout the rest of the book. The assumption that the characteristic is not $2$ is common when studying quadratic forms and related objects, such as symmetric bilinear forms and central simple algebras with orthogonal involution, since in this case there is the well known bijective correspondence between quadratic forms and their polar symmetric bilinear forms. However, this correspondence does not hold when the characteristic is $2$, and so to handle this case the authors of \cite{KMRT} introduce and develop the notion of a \emph{quadratic pair}. Given a central simple algebra $A$ over a field $\FF$ of any characteristic, a quadratic pair on $A$ is $(\sigma,f)$ where $\sigma$ is an orthogonal involution and $f\colon \Sym(A,\sigma) \to \FF$ is a linear function satisfying $f(a+\sigma(a)) = \Trd_A(a)$ for all $a\in A$. Here $\Sym(A,\sigma)$ are the symmetric elements of $A$ and $\Trd_A$ is the reduced trace. Such functions $f$ are called \emph{semi-traces}. These pairs $(\sigma,f)$ play the role of orthogonal involutions in characteristic $2$, for example the group of algebra automorphism of $A$ which preserve $\sigma$ may fail to be smooth, but those which preserve $(\sigma,f)$ will be a smooth semisimple group of type $D$. 

The exposition of triality in \cite{KMRT} (assuming characteristic not $2$) relies on the fact that for a central simple algebra of degree $8$ with orthogonal involution $(A,\sigma)$ of trivial discriminant, the Clifford algebra decomposes as a Cartesian product $\mathrm{Cl}(A,\sigma) \cong (B,\sigma_B)\times(C,\sigma_C)$ of two central simple algebras of degree $8$, each with an orthogonal involution of trivial discriminant. However, when discussing triality in characteristic $2$, the algebra $A$ will have a quadratic pair, say $(A,\sigma,f)$, and the appropriate objects to appear on the right should be algebras with quadratic pairs as well, $(B,\sigma_B,f_B)$ and $(C,\sigma_C,f_C)$. Therefore, one should also have a quadratic pair on the Clifford algebra. Triality in characteristic $2$ was avoided by the authors of \cite{KMRT} since, in their words, ``we did not succeed in giving a rational definition of the quadratic pair on $\mathrm{Cl}(A,\sigma,f)$" \cite[p. 571]{KMRT}. Recently in \cite{DQ21}, Dolphin and Qu\'eguiner-Mathieu gave a rational definition of the canonical quadratic pair on the Clifford algebra when working over a field of characteristic $2$. Their construction uses an element $a\in A$ with $\Trd_A(a)=1$ and defines the quadratic pair using the image of $a$ under the canonical mapping $c \colon A \to \mathrm{Cl}(A,\sigma,f)$. 

In this paper, we generalize the construction of \cite{DQ21} to define the canonical quadratic pair on the Clifford algebra when working over an arbitrary base scheme $S$. The notion of quadratic pair was generalized to the setting over a scheme in \cite{CF}, and then further studied in \cite{GNR}. In this setting we work primarily with sheaves on the category $\Sch_S$ of schemes over $S$ equipped with the fppf topology. The role of the base field is played by the sheaf of global sections $\cO$, which sends a scheme $T \mapsto \cO_T(T)$, and central simple algebras are replaced by Azumaya algebras, which are sheaves locally isomorphic to finite rank matrix algebras over $\cO$. Unlike in the setting over fields, or even over rings, an Azumaya algebra $\cA$ may fail to have a global section $a\in \cA(S)$ for which $\Trd_{\cA}(a) = 1$, and therefore some modification of the definition from \cite{DQ21} is needed. 

Let $(\cA,\sigma,f)$ be an Azumaya algebra of even rank $2n$ with a quadratic pair. As we review in the preliminaries, the Clifford algebra $\Cl(\cA,\sigma,f)$, now denoted in cursive since it is also a sheaf, has a canonical involution $\underline{\sigma}$ which is orthogonal if and only if $n \equiv 0 \pmod{4}$ or both $n \equiv 2 \pmod{4}$ and $2=0 \in \cO(S)$. Assume we are in one of these cases. Let
\begin{align*}
\cSymd_{(\Cl(\cA,\sigma,f),\underline{\sigma})} &= \Img(\Id+\underline{\sigma}) \subseteq \Cl(\cA,\sigma,f) \text{, and}\\
\cAlt_{(\Cl(\cA,\sigma,f),\underline{\sigma})} &= \Img(\Id - \underline{\sigma}) \subseteq \Cl(\cA,\sigma,f)
\end{align*}
be the image subsheaves of symmetrized and alternating elements respectively. Since alternating elements are skew-symmetric, we have a commutative diagram
\[
\begin{tikzcd}
\Cl(\cA,\sigma,f) \arrow{r}{\Id + \underline{\sigma}} \arrow{d} & \cSymd_{(\Cl(\cA,\sigma,f),\underline{\sigma})} \arrow[equals]{d} \\
\Cl(\cA,\sigma,f)/\cAlt_{(\Cl(\cA,\sigma,f),\underline{\sigma})} \arrow{r}{\xi} & \cSymd_{(\Cl(\cA,\sigma,f),\underline{\sigma})}
\end{tikzcd}
\]
where $\xi$ is the induced map. We use \cite[4.17(b)]{GNR} which states that a quadratic pair $(\underline{\sigma},f_{\ell})$ may be constructed from any $\ell \in (\Cl(\cA,\sigma,f)/\cAlt_{(\Cl(\cA,\sigma,f),\underline{\sigma})})(S)$ for which $\xi(\ell)=1_{\cA}$. Our construction of the canonical quadratic pair is as follows.

\begin{thm}\label{intro_thm_construction}
Let $(\cA,\sigma,f)$ be an Azumaya algebra of degree $2n$ with quadratic pair such that $(\Cl(\cA,\sigma,f),\underline{\sigma})$ is an algebra with orthogonal involution. Assume that $n\geq 4$. Then, there is a commutative diagram
\[
\begin{tikzcd}[column sep=3ex]
0 \arrow{r} & \slLie_{\cA} \arrow[hookrightarrow]{r} \arrow{d}{c} & \cA \arrow{r}{\Trd_{\cA}} \arrow{d}{c} & \cO \arrow{d}{\rho} \arrow{r} & 0 \\
0 \arrow{r} & \cAlt_{(\Cl(\cA,\sigma,f),\underline{\sigma})} \arrow[hookrightarrow]{r} & \Cl(\cA,\sigma,f) \arrow{r} & \Cl(\cA,\sigma,f)/\cAlt_{(\Cl(\cA,\sigma,f),\underline{\sigma})} \arrow{r} & 0
\end{tikzcd}
\]
where $\slLie_{\cA}$ is the submodule of trace $0$ sections and $c\colon \cA \to \Cl(\cA,\sigma,f)$ is the canonical mapping. The element $\rho(1_{\cO})$ satisfies $\xi(\rho(1_{\cO})) = 1_{\cA}$ and therefore we define $\underline{f}=f_{\rho(1_{\cO})}$ which gives the \emph{canonical quadratic pair} on the Clifford algebra $(\Cl(\cA,\sigma,f),\underline{\sigma},\underline{f})$.

The canonical quadratic pair has the property that an isomorphism of algebras with quadratic pairs $(\cA,\sigma,f) \iso (\cA',\sigma',f')$ induces an isomorphism of Clifford algebras with canonical quadratic pairs
\[
(\Cl(\cA,\sigma,f),\underline{\sigma},\underline{f}) \iso (\Cl(\cA',\sigma',f'),\underline{\sigma'},\underline{f'}).
\]

Furthermore, the construction is stable under base change and therefore the canonical group homomorphism $\PGO_{(\cA,\sigma,f)} \to \bAut(\Cl(\cA,\sigma,f),\underline{\sigma})$ factors through $\bAut(\Cl(\cA,\sigma,f),\underline{\sigma},\underline{f})$, yielding a homomorphism
\[
\PGO_{(\cA,\sigma,f)} \to \bAut(\Cl(\cA,\sigma,f),\underline{\sigma},\underline{f}).
\]
\end{thm}

\Cref{intro_thm_construction} covers all cases when $\underline{\sigma}$ is orthogonal except the case when $n=2$ and $2=0 \in \cO(S)$. We also address this case, however we show that no canonical quadratic pair exists.
\begin{thm}\label{intro_thm_degree_4}
Assume that $2=0 \in \cO(S)$ and let $(\cA,\sigma,f)$ be an Azumaya algebra of degree $4$ with quadratic pair. Then, $(\Cl(\cA,\sigma,f),\underline{\sigma})$ is an algebra with orthogonal involution, however it cannot be equipped with a canonical quadratic pair. By which we mean that, for any choice of semi-trace $f' \colon \cSym_{(\Cl(\cA,\sigma,f),\underline{\sigma})} \to \cO$, the canonical group homomorphism $\PGO_{(\cA,\sigma,f)} \to \bAut(\Cl(\cA,\sigma,f),\underline{\sigma})$ does not factor through $\bAut(\Cl(\cA,\sigma,f),\underline{\sigma},f')$.
\end{thm}

The paper is arranged as follows. \Cref{prelim} reviews the technicalities of our setting over a scheme and reviews the definitions and basic properties of the objects we will work with. In particular, \Cref{prelim_Clifford} contains a detailed description of the split Clifford algebra in \Cref{split_Clifford}, which we use later for explicit calculations. This example also includes a description of a natural quadratic form $q_\wedge$ whose adjoint involution is the canonical involution of the Clifford algebra. \Cref{Clifford_triples} contains the main results. The many parts of \Cref{intro_thm_construction} are found throughout \Cref{A_degree_8}. \Cref{sl_into_Alt} shows that $c\colon A \to \Cl(\cA,\sigma,f)$ restricts to $c\colon \slLie_{\cA} \to \cAlt_{(\Cl(\cA,\sigma,f),\underline{\sigma})}$, the diagram \eqref{eq_rho_diagram} defines $\rho$, \Cref{rho_xi_compatible} shows that $\xi(\rho(1_{\cO}))=1_{\cA}$, and the final claims about isomorphisms and the group homomorphism are \Cref{canonical_functorial} and \Cref{canonical_group_sheaf_morphism} respectively. We also show in \Cref{canonical_Phi_correspondence} that in the split case the canonical quadratic pair corresponds to the natural quadratic form $q_\wedge$ of \Cref{split_Clifford}. Finally, the short \Cref{A_degree_4} contains some preliminary lemmas involving calculations with the split Clifford algebra before proving \Cref{intro_thm_degree_4}, which is \Cref{degree_4_no_canonical}.
}

\section{Preliminaries}\label{prelim}
\subsection{Setting over a Scheme}
We work over an arbitrary fixed base scheme $S$. Most objects will be sheaves on the category of schemes over $S$, $\Sch_S$, equipped with the fppf topology. In particular, we work with a big fppf site as in \cite[Tag 021R]{Stacks}. If $\{T_i \to T\}_{i\in I}$ is an fppf cover, we will use the notation $T_{ij}=T_i \times_T T_j$. We also refer to $\Aff_S$, the big affine fppf site over $S$ of \cite[Tag 021S (2)]{Stacks}, commonly denoting affine schemes with $U$ or $V$. Fixing some notation, we let
\begin{enumerate}[label={\rm(\roman*)}]
\item $\Sets$ be the category of sets,
\item $\Grp$ be the category of groups,
\item $\Ab$ be the category of abelian groups, and
\item $\Rings$ be the category of unital commutative associative rings.
\end{enumerate}
If $\cF$ is a sheaf on $\Sch_S$, by abuse of notation we will refer to ``a section $f\in \cF$" by which we mean a section $f\in \cF(T)$ for some $T\in \Sch_S$. Similarly, a phrase like ``for appropriate sections $f \in \cF$ and $f'\in \cF'$" will mean that $f \in \cF(T)$ and $f'\in \cF'(T)$ are both sections over some common $T \in \Sch_S$. If $T\in \Sch_S$, then $\cF|_T$ will be the sheaf $\cF$ restricted to the site $\Sch_T$. If $T'\to T$ is a morphism in $\Sch_S$ and $f \in \cF(T)$, then $f|_{T'}$ will denote the image of $f$ under the restriction map $\cF(T) \to \cF(T')$.

When working with modules and algebras, they will be with respect to the \emph{structure sheaf} of \cite[Tag 03DU]{Stacks},
\begin{align*}
\cO \colon \Sch_S &\to \Rings \\
T &\mapsto \Gamma(T,\cO_T),
\end{align*}
which sends a scheme to its ring of global sections. We then work on the ringed site $(\Sch_S,\cO)$ with $\cO$--modules as in \cite[Tag 03CW]{Stacks} or with $\cO$--algebras, which are $\cO$--modules equipped with a $\cO$--bilinear multiplication. 

Given two $\cO$--modules $\cM_1$ and $\cM_2$, their sheaf of internal homomorphisms
\begin{align*}
\cHom_{\cO}(\cM_1,\cM_2) \colon \Sch_S &\mapsto \Ab \\
T &\mapsto \Hom_{\cO|_T}(\cM_1|_T,\cM_2|_T)
\end{align*}
is again an $\cO$--module. We write $\cEnd_{\cO}(\cM) = \cHom_{\cO}(\cM,\cM)$, and for the subsheaf of automorphisms we write $\bAut_{\cO}(\cM)$. For an $\cO$--module $\cM$, its \emph{dual} is $\cM^* = \cHom_{\cO}(\cM,\cO)$.

Following \cite[Tag 03DL (2)]{Stacks} and applying \cite[Tag 03DN]{Stacks} because $S\in \Sch_S$ is a final object, we call an $\cO$--module $\cM$ \emph{finite locally free} if there exists an fppf cover $\{T_i \to S\}_{i\in I}$ such that for each $i\in I$ there is an isomorphism $\cM|_{T_i}\cong \cO^{n_i}$ for some $n_i \in \NN$. We say $\cM$ is of \emph{everywhere positive rank} if all $n_i > 0$, and we say it is of \emph{constant rank $d$} if all $n_i = d$. If a module is both finite locally free and everywhere of positive rank we will call it positive finite locally free.

Let $\cA$ be an $\cO$--algebra and let $\underline{\cA}$ be the underlying $\cO$--module. The algebra $\cA$ is an \emph{Azumaya $\cO$--algebra} (or simply \emph{Azumaya algebra}) if it is positive finite locally free and satisfies the following equivalent conditions.
\begin{enumerate}[label={\rm(\roman*)}]
\item The enveloping morphism
\begin{align}
\varphi_e \colon \cA \otimes_{\cO} \cA\op &\iso \cEnd_{\cO}(\underline{\cA}) \label{eq_enveloping} \\
a\otimes b\op &\mapsto (x \mapsto axb) \nonumber
\end{align}
is an isomorphism of $\cO$--algebras.
\item For any affine scheme $U\in \Aff_S$, $\cA(U)$ is an Azumaya $\cO(U)$--algebra.
\end{enumerate}
These conditions come from \cite[5.1]{Gro} and \cite[2.5.3.4]{CF} respectively. For Azumaya algebras over rings we refer to \cite{Ford}. Since Azumaya algebras over rings become matrix algebras after faithfully flat extension, we also have the following conditions equivalent to the first two.
\begin{enumerate}[label={\rm(\roman*)}]
\setcounter{enumi}{2}
\item There exists a cover $\{T_i \to S\}_{i\in I}$ such that for each $i\in I$, $\cA|_{T_i} \cong \cEnd_{\cO|_{T_i}}(\cM_i)$ for some $\cO|_{T_i}$--module $\cM_i$ which is positive finite locally free.
\item There exists a cover $\{T_i \to S\}_{i\in I}$ such that for each $i\in I$, $\cA|_{T_i} \cong \Mat_{n_i}(\cO|_{T_i})$ for some $n_i>0$.
\end{enumerate}

An Azumaya $\cO$--algebra $\cA$ is equipped with a linear map $\Trd_{\cA} \colon \cA \to \cO$ called the \emph{reduced trace}, which is locally the trace of matrix algebras. We set $\slLie_{\cA} = \Ker(\Trd_{\cA})$ to be the submodule of trace zero sections.

Let $\cB$ be an $\cO$--algebra which is locally a finite cartesian product of Azumaya algebras. In particular, there exists a cover $\{T_i \to S\}_{i\in I}$ such that
\[
\cB|_{T_i} \cong \cA_{i,1}\times \ldots \times \cA_{i,n}
\]
for Azumaya $\cO|_{T_i}$--algebras $\cA_{i,1},\ldots,\cA_{i,n}$. The local linear maps
\[
\Trd_{\cA_{i,1}} + \ldots + \Trd_{\cA_{i,n}} \colon \cA_{i,1}\times \ldots \times \cA_{i,n} \to \cO|_{T_i}
\]
will glue into a linear map denoted $\Trd_{\cB} \colon \cB \to \cO$ and also called the reduced trace of $\cB$.

\subsection{Quadratic Forms and Quadratic Triples}\label{quad_forms_and_triples}
For the theory of quadratic forms over rings we refer to \cite{Knus}. For quadratic forms and quadratic triples over schemes we refer to \cite{CF} and \cite{GNR}.

A \emph{quadratic form} is a natural transformation $q\colon \cM \to \cO$ from an $\cO$--module $\cM$ such that
\begin{enumerate}[label={\rm(\roman*)}]
\item $q(cm)=c^2q(m)$ for all appropriate sections $c\in \cO$ and $m\in \cM$, and
\item the polar form $b_q(m_1,m_2)=q(m_1+m_2)-q(m_1)-q(m_2)$ is an $\cO$--bilinear form $b_q \colon \cM \times \cM \to \cO$.
\end{enumerate}
The pair $(\cM,q)$ may also be called a \emph{quadratic module}.

A bilinear form $b\colon \cM \times \cM \to \cO$ is called \emph{regular} if the map
\begin{align*}
b^* \colon \cM &\to \cM^* \\
m &\mapsto b(m,\und)
\end{align*}
is an isomorphism. A quadratic form $q$ is called \emph{regular} if $b_q$ is regular. A regular bilinear form has an adjoint involution $\sigma \colon \cEnd_{\cO}(\cM) \to \cEnd_{\cO}(\cM)$ defined uniquely by requiring that
\[
b(m_1,B(m_2)) = b(\sigma(B)(m_1),m_2)
\]
for all appropriate sections $m_1,m_2 \in \cM$ and $B\in \cEnd_{\cO}(\cM)$. A regular bilinear form also induces an isomorphism
\begin{align}
\varphi_b \colon \cM\otimes_{\cO}\cM &\iso \cEnd_{\cO}(\cM) \label{eq_bilinear_iso} \\
m_1\otimes m_2 &\mapsto b(m_1,\und)\cdot m_2. \nonumber
\end{align}

If $\cA$ is an Azumaya $\cO$--algebra, an involution $\sigma$ on $\cA$ will always mean an $\cO$--linear anti-automorphism. Sufficiently locally, any involution $\sigma$ is the adjoint involution of some regular bilinear forms. We call $\sigma$
\begin{enumerate}[label={\rm(\roman*)}]
\item \emph{orthogonal} if the bilinear forms are symmetric,
\item \emph{weakly-symplectic} if the bilinear forms are skew-symmetric, and
\item \emph{symplectic} if the bilinear forms are alternating.
\end{enumerate}
Symplectic involutions are always weakly-symplectic. If $\frac{1}{2}\in \cO$, then weakly-symplectic and symplectic are the same and they are disjoint from orthogonal. However, in the case when $2=0\in \cO$, the notions of orthogonal and weakly-symplectic coincide.

Let $(\cA,\sigma)$ be an Azumaya $\cO$--algebra with involution. We have two maps $\Id \pm \sigma \colon \cA \to \cA$ and we define the submodules
\begin{align*}
\cSym_{(\cA,\sigma)} &= \Ker(\Id - \sigma) & \cSkew_{(\cA,\sigma)} &= \Ker(\Id + \sigma) \\
\cSymd_{(\cA,\sigma)}&= \Img(\Id+\sigma) & \cAlt_{(\cA,\sigma)} &= \Img(\Id - \sigma)
\end{align*}
where $\cSymd_{(\cA,\sigma)}$ and $\cAlt_{(\cA,\sigma)}$ are image sheaves and therefore involve sheafification. On the left, these are the \emph{symmetric} and \emph{symmetrized} elements, while on the right we have the \emph{skew-symmetric} and \emph{alternating} elements.

A \emph{quadratic triple} is $(\cA,\sigma,f)$ where $(\cA,\sigma)$ is an Azumaya $\cO$--algebra with orthogonal involution and $f\colon \cSym_{(\cA,\sigma)} \to \cO$ is an $\cO$--linear map satisfying
\[
f(a+\sigma(a)) = \Trd_{\cA}(a)
\]
for all $a\in \cA$. This is also referred to as a \emph{quadratic pair} $(\sigma,f)$ on $\cA$. The function $f$ is called a \emph{semi-trace}. A morphism of quadratic triples $\varphi \colon (\cA,\sigma,f) \to (\cA',\sigma',f')$ is a morphism $\varphi \colon \cA \to \cA'$ between the Azumaya algebras such that $\varphi \circ \sigma = \sigma'\circ \varphi$ and $f'\circ \varphi = f$.

By \cite[3.6(iii)]{GNR}, if $\sigma$ is an orthogonal involution, then the submodules $\cSym_{(\cA,\sigma)}$ and $\cAlt_{(\cA,\sigma)}$ are mutually orthogonal with respect to the trace form $(a,b)\mapsto \Trd_{\cA}(ab)$ on $\cA$. Therefore, any element $\lambda \in (\cA/\cAlt_{(\cA,\sigma)})(S)$ gives a well defined map $\Trd_{\cA}(\lambda \cdot \und)\colon \cSym_{(\cA,\sigma)} \to \cO$. By \cite[4.18]{GNR}, given a quadratic triple $(\cA,\sigma,f)$, there exists a unique element $\lambda_f\in (\cA/\cAlt_{(\cA,\sigma)})(S)$ such that $f = \Trd_{\cA}(\lambda_f \cdot \und)$.

If $(\cM,q)$ is a regular quadratic module, then by \cite[4.4]{GNR} it corresponds to a unique quadratic triple $(\cEnd_{\cO}(\cM),\sigma_q,f_q)$ where $\sigma_q$ is the adjoint involution of $b_q$ and the semi-trace satisfies
\[
f_q(\varphi_b(m\otimes m))=q(m)
\]
for all $m\in \cM$. Conversely, let $b\colon \cM\times \cM \to \cO$ be a regular symmetric bilinear form with adjoint involution $\sigma_b$. If $(\cEnd_{\cO}(\cM),\sigma_b,f)$ is a quadratic triple then it corresponds to a unique quadratic form
\[
q_f(m) = f(\varphi_b(m\otimes m))
\]
on $\cM$ and the polar of $q_f$ is $b$. The triple $(\cEnd_{\cO}(\cM),\sigma_q,f_q)$ is called the \emph{adjoint} of $(\cM,q)$.

Let $\cB$ be an $\cO$--algebra which is locally a Cartesian product of Azumaya algebras. An \emph{orthogonal involution} on $\cB$ will be an order two anti-automorphism of $\cB$ which acts trivially on the center, is therefore locally a product of involutions on each Azumaya algebra, and those local involutions are orthogonal. A \emph{weakly-symplectic} or \emph{symplectic involution} on $\cB$ is defined similarly. Let $(\cB,\sigma)$ be such an algebra with orthogonal involution. A linear map
$f\colon \cSym_{(\cB,\sigma)} \to \cO$ satisfying $f(b+\sigma(b)) = \Trd_{\cB}(b)$ for all $b\in \cB$ will also be called a \emph{semi-trace}. The triple $(\cB,\sigma,f)$ will be called an \emph{algebra with semi-trace}. Our terminology differs slightly from \cite{DQ21}, where $(\sigma,f)$ would also be called a quadratic pair on $\cB$ despite $\cB$ not being an Azumaya $\cO$--algebra. Such an algebra with semi-trace will locally be a cartesian product of quadratic triples. The process of defining a semi-trace using an element $\lambda \in \cB/\cAlt_{(\cB,\sigma)}$ detailed above works analogously here by virtue of local considerations.

\subsection{Clifford Algebras over Schemes}\label{prelim_Clifford}
We recall the definition of the Clifford algebras associated to a quadratic form $(\cM,q)$ and to a quadratic triple $(\cA,\sigma,f)$ over our base scheme $S$ from \cite[4.2]{CF}. 

\begin{defn}\label{defn_form_Clifford}
Let $(\cM,q)$ be an $\cO$--module with a quadratic form and let $\cT(\cM)$ be the tensor $\cO$--algebra. Then, the \emph{Clifford algebra} of $(\cM,q)$ is
\[
\Cl(\cM,q) = \cT(\cM)/\cI
\]
where $\cI$ is the subsheaf of two-sided ideals generated by elements of the form $m\otimes m - q(m)$.
\end{defn}
Analogous to the theory of Clifford algebras over fields or rings, $\Cl(\cM,q)$ decomposes into even and odd components, denoted $\Cl_0(\cM,q)$ and $\Cl_1(\cM,q)$ respectively. It also has a submodule isomorphic to $\cM$ which is the image of the map $m\mapsto m$, and a canonical involution $\underline{\sigma}$ which is the identity on $\cM$. This involution also restricts to the even and odd portions of the Clifford algebra and is denoted $\underline{\sigma}_0$ and $\underline{\sigma}_1$ respectively. By \cite[4.2.0.7]{CF}, if $U\in \Aff_S$, then $\Cl(\cM,q)(U) \cong \mathrm{Cl}(\cM(U),q(U))$ where the right hand side is the Clifford algebra over the ring $\cO(U)$ as in \cite[IV.1.1.2]{Knus}.

Next, let $(\cA,\sigma,f)$ be a quadratic triple over $S$. Let $\underline{\cA}$ be the $\cO$--module underlying $\cA$ and let $\cT(\underline{\cA})$ be the tensor algebra. Since $\cA$ is an Azumaya $\cO$--algebra, the enveloping map $\varphi_e$ of \eqref{eq_enveloping} is an algebra isomorphism, so in particular gives an $\cO$--module isomorphism
\[
\underline{\varphi_e} \colon \underline{\cA}\otimes_{\cO} \underline{\cA} \iso \underline{\cEnd_{\cO}(\underline{\cA})}.
\]
Now, for $T\in \Sch_S$ and a section $u\in (\underline{\cA}\otimes_{\cO} \underline{\cA})(T)$, the map
\begin{align*}
\underline{\cA}|_T &\mapsto \underline{\cA}|_T \\
x &\mapsto \underline{\varphi_e}(u)|_T(\sigma(x))
\end{align*}
is an $\cO|_T$--module morphism and is therefore equal to $\underline{\varphi_e}(v)$ for some unique $v\in (\underline{\cA}\otimes_{\cO} \underline{\cA})(T)$. Setting $\sigma_2(u)=v$ defines an order two automorphism $\sigma_2 \colon \underline{\cA}\otimes_{\cO} \underline{\cA} \to \underline{\cA}\otimes_{\cO} \underline{\cA}$.

At this point we deviate slightly from the exposition in \cite{CF}, but utilize an equivalent construction.
\begin{lem}
Let $u\in \underline{\cA}\otimes_{\cO} \underline{\cA}$ with $\sigma_2(u)=u$. Then $(\underline{\varphi_e}(u))(a)=0$ for all appropriate sections $a\in \cAlt_{(\cA,\sigma)}$.
\end{lem}
\begin{proof}
Since $a \in \cAlt_{(\cA,\sigma)}$, it is locally of the form $a'-\sigma(a')$. We can then compute that $(\underline{\varphi_e}(u))(a)$ is locally of the form
\begin{align*}
(\underline{\varphi_e}(u))(a'-\sigma(a'))&=(\underline{\varphi_e}(u))(a')-(\underline{\varphi_e}(u))(\sigma(a')) \\
&=(\underline{\varphi_e}(u))(a')-(\underline{\varphi_e}(\sigma_2(u)))(a') = 0
\end{align*}
and therefore $(\underline{\varphi_e}(u))(a)=0$ as claimed.
\end{proof}
This means that for every $u\in \underline{\cA}\otimes_{\cO} \underline{\cA}$ with $\sigma_2(u)=u$, the map $\underline{\varphi_e}(u)$ descends to a well defined function
\[
\phi_u \colon \cA/\cAlt_{(\cA,\sigma)} \to \underline{\cA}.
\]
We use these functions to define the Clifford algebra.
\begin{defn}\label{defn_triple_Clifford}
Let $(\cA,\sigma,f)$ be a quadratic triple over $S$ and let $\lambda_f \in (\cA/\cAlt_{(\cA,\sigma)})(S)$ be the unique element corresponding to $f$ as in \cite[4.18]{GNR}. The \emph{Clifford algebra} of $(\cA,\sigma,f)$ is
\[
\Cl(\cA,\sigma,f) = \frac{\cT(\underline{\cA})}{\cJ_1(\sigma,f)+\cJ_2(\sigma,f)}
\]
where
\begin{enumerate}[label={\rm(\roman*)}]
\item $\cJ_1(\sigma,f)$ is the subsheaf of two-sided ideals generated by all sections of the form $s-f(s)$ for $s\in \cSym_{(\cA,\sigma)}$, and
\item $\cJ_2(\sigma,f)$ is the subsheaf of two-sided ideals generated by all sections of the form $u-\phi_u(\lambda_f)$ where $u\in \underline{\cA}\otimes_{\cO}\underline{\cA}$ with $\sigma_2(u)=u$.
\end{enumerate}
\end{defn}
We have a composition of canonical $\cO$--module maps
\[
\cA \inj \underline{\cA} \inj \cT(\underline{\cA}) \to \Cl(\cA,\sigma,f),
\]
which we denote by
\begin{equation}\label{eq_c_A_inclusion}
c \colon \cA \to \Cl(\cA,\sigma,f)
\end{equation}
and call the \emph{canonical mapping}. The image of $c$ generates $\Cl(\cA,\sigma,f)$ as an algebra. The Clifford algebra is also equipped with a canonical involution $\underline{\sigma}$ which is defined by $\underline{\sigma}(c(a)) = c(\sigma(a))$. The following theorem is a key fact.
\begin{thm}[{\cite[4.2.0.14]{CF}}] \label{Clifford_iso}
If $(\cM,q)$ is a regular quadratic module of constant even rank, then there is an isomorphism of algebras with involution
\[
\Phi \colon (\Cl_0(\cM,q),\underline{\sigma}_0) \iso (\Cl(\cEnd_{\cO}(\cM),\sigma_q,f_q),\underline{\sigma_q})
\]
This isomorphism is induced by the isomorphism of modules $\varphi_b \colon \cM\otimes_{\cO} \cM \iso \cEnd_{\cO}(\cM)$ of \eqref{eq_bilinear_iso}.
\end{thm}

The Clifford algebra construction is functorial. Given a homomorphism of quadratic triples $(\cA,\sigma,f) \to (\cA',\sigma',f')$ there is a homomorphism of algebras with involution
\begin{equation}\label{eq_Clifford_functorial}
\Cl(\varphi)\colon (\Cl(\cA,\sigma,f),\underline{\sigma}) \to (\Cl(\cA',\sigma',f'),\underline{\sigma'})
\end{equation}
defined by $\bC(\varphi)(c(a)) = c'(\varphi(a))$ for all $a\in \cA$ where $c\colon \cA \to \Cl(\cA,\sigma,f)$ and $c'\colon \cA' \to \Cl(\cA',\sigma',f')$ are the respective canonical mappings.

In what follows, we will use the Clifford algebra of a hyperbolic quadratic form for calculations. We outline the construction of this algebra and some of its key properties in the following example.
\begin{ex}\label{split_Clifford}
We let $\cV = \cO^n$ be the free $\cO$--module of rank $n$, and set $\HH(\cV) = \cV \oplus \cV^*$. Let $\{v_1,\ldots,v_n\}$ be the standard basis of $\cV$, and $\{v_1^*,\ldots,v_n^*\}$ the dual basis of $\cV^*$. We define a hyperbolic quadratic form on $\HH(\cV)$ by setting $q_{2n}(x+g)=g(x)$ for $x\in \cV$ and $g\in \cV^*$. This is a regular quadratic form. The Clifford algebra $\Cl(\HH(\cV),q_{2n})$ is generated by the elements $v_i$ and $v_j^*$ subject to the relations
\begin{align*}
v_i^2 &= 0 & (v_i^*)^2 &= 0 \\
v_iv_j &= -v_jv_i & v_i^* v_j^* &= -v_j^* v_i^* \\
v_iv_i^* &= 1-v_i^*v_i & v_iv_j^* &= -v_j^*v_i
\end{align*}
for all $1\leq i,j \leq n$ with $i\neq j$.

Let $\wedge \cV$ be the exterior algebra of $\cV$, and let $\wedge_0 \cV$ and $\wedge_1 \cV$ be the even and odd parts respectively. We define two families of endomorphisms of $\wedge\cV$. For $x\in \cV$, let $\ell_x \colon \wedge \cV \to \wedge\cV$ be the left multiplication by $x$. For $g\in \cV^*$, define $d_g \colon \wedge \cV \to \wedge \cV$ by
\[
d_g(x_1\wedge \ldots \wedge x_k) = \sum_{i=1}^k (-1)^{i+1}g(x_i)\cdot x_1\wedge\ldots\wedge \hat{x_i} \wedge \ldots \wedge x_k
\]  
where hat denotes omission. Then by \cite[4.2.0.10]{CF} (which follows from \cite[IV.2.1.1]{Knus}), there is an isomorphism of algebras
\begin{align*}
\Phi \colon \Cl(\HH(\cV),q_{2n}) &\iso \cEnd_{\cO}(\wedge \cV) \\
v_i &\mapsto \ell_{v_i} \\
v_i^* &\mapsto d_{v_i^*}
\end{align*}
which restricts to an isomorphism
\begin{equation}\label{eq_Phi_not}
\Phi_0 \colon \Cl_0(\HH(\cV),q_{2n}) \iso \cEnd_{\cO}(\wedge_0 \cV)\times \cEnd_{\cO}(\wedge_1 \cV).
\end{equation}
To track the canonical involution across this isomorphism, we equip $\wedge \cV$ with its own bilinear form following \cite[pg.90]{KMRT}. Let $[n]=\{1,\ldots,n\}$ be the set of integers from $1$ to $n$, and then for $I=\{i_1,\ldots,i_k\}\subseteq [n]$, written in increasing order, we set $v_I = v_{i_1}\wedge v_{i_2} \wedge \ldots \wedge v_{i_k}$. These $v_I$ form a basis of $\wedge\cV$. Now, for $x\in \wedge\cV$, let $\overline{x}$ be the element with the wedge factors reversed, i.e., $\overline{x_1\wedge x_2 \wedge \ldots \wedge x_k}=x_k\wedge \ldots \wedge x_2 \wedge x_1$ and extended linearly. For $I\subseteq [n]$ we have that
\[
\overline{v_I} = (-1)^{\frac{(|I|-1)|I|}{2}}v_I = \begin{cases} v_I & |I|\equiv 0,1 \pmod{4} \\ -v_I &|I|\equiv 2,3 \pmod{4}. \end{cases}
\]
Let $\pi \colon \wedge\cV \to \cO$ be the map coming from the projection onto the submodule spanned by $e_{[n]}$. Define a bilinear form $b_\wedge \colon \wedge \cV \times \wedge\cV \to \cO$ by
\[
b_{\wedge}(x,y) = \pi(\overline{x}\wedge y).
\]
\begin{lem}\label{b_wedge_formula}
We have the following formula for the bilinear form $b_\wedge$.
\[
b_\wedge\left(\sum_{I\subseteq [n]}\alpha_I v_I , \sum_{I\subseteq [n]}\beta_I v_I\right) = \sum_{I\subseteq [n]} (-1)^{(\Sigma I)-|I|} \alpha_I \beta_{I^c}.
\]
where $\Sigma I = \sum_{i\in I} i$ and $\alpha_I,\beta_I \in \cO$ are coefficients.
\end{lem}
\begin{proof}
First, we argue that $\overline{v_I}\wedge v_{I^c} = (-1)^{(\Sigma I)-|I|} v_{[n]}$. Let $I=\{i_1,\ldots,i_k\}$, written in increasing order. Then
\[
\overline{v_I}\wedge v_{I^c} = v_{i_k}\wedge \ldots \wedge v_{i_1}\wedge v_{I^c}.
\]
If we commute the leading factor $v_{i_k}$ to the right until it is in the correct position, it must pass exactly those $v_{\ell}$ (either in $\overline{v_I}$ or in $v_{I^c}$) for which $\ell < i_k$. Therefore, this process produces a factor of $(-1)^{i_k-1}$ and we have
\[
\overline{v_I}\wedge v_{I^c} = (-1)^{i_k - 1} \overline{v_{I\backslash \{v_{i_k}\}}} \wedge v_{I^c \cup \{v_{i_k}\}}.
\]
Repeating this process for all factors in $v_I$, we obtain
\[
\overline{v_I}\wedge v_{I^c} = (-1)^{\sum_{i\in I}(i-1)} v_{[n]} = (-1)^{\Sigma I - |I|}v_{[n]}.
\]
Now we may simply compute.
\begin{align*}
b_\wedge\left(\sum_{I\subseteq [n]}\alpha_I v_I , \sum_{I\subseteq [n]}\beta_I v_I\right) &= \pi\left( \left( \sum_{I\subseteq [n]}\alpha_I \overline{v_I}\right)\wedge \left( \sum_{I\subseteq [n]}\beta_I v_I\right)\right) \\
&= \pi\left( \sum_{I\subseteq [n]} \alpha_I \beta_{I^c} \overline{v_I}\wedge v_{I^c} \right) \\
&= \sum_{I\subseteq [n]} (-1)^{(\Sigma I)-|I|} \alpha_I \beta_{I^c}.
\end{align*}
In the second equality, we both expanded the product of the two sums and dropped any term for which $\pi(\overline{v_I}\wedge v_J)=0$, which are those where $J\neq I^c$. This justifies the claimed formula.
\end{proof}

\begin{lem}\label{b_wedge_properties}
We have the following.
\begin{enumerate}[label={\rm(\roman*)}]
\item \label{b_wedge_properties_i} $b_\wedge$ is symmetric when $n\equiv 0,1 \pmod{4}$ and it is alternating when $n\equiv 2,3 \pmod{4}$. 
\item \label{b_wedge_properties_ii} $b_\wedge$ is regular. 
\item \label{b_wedge_properties_iii} Since $b_\wedge$ is regular, it has an adjoint involution $\tau$. This involution corresponds to $\underline{\sigma}$ via $\Phi$, i.e., $\tau = \Phi\circ\underline{\sigma}\circ \Phi^{-1}$.
\end{enumerate}
\end{lem}
\begin{proof}
\noindent\ref{b_wedge_properties_i}: Using \Cref{b_wedge_formula} but reversing the arguments yields
\begin{align*}
b_\wedge\left(\sum_{I\subseteq [n]}\beta_I v_I, \sum_{I\subseteq [n]}\alpha_I v_I \right) &= \sum_{I\subseteq [n]} (-1)^{(\Sigma I)-|I|} \beta_I \alpha_{I^c} \\
&= \sum_{I\subseteq [n]} (-1)^{(\Sigma I^c)-|I^c|} \alpha_I \beta_{I^c}.
\end{align*}
Hence we need to compare the coefficients $(-1)^{(\Sigma I)-|I|}$ and $(-1)^{(\Sigma I^c)-|I^c|}$. Their product is
\[
(-1)^{(\Sigma I)-|I|}\cdot (-1)^{(\Sigma I^c)-|I^c|} = (-1)^{\Sigma [n] -n} = (-1)^{\Sigma[n-1]} = (-1)^{\frac{n(n-1)}{2}}
\]
which is $1$ if $n\equiv 0,1 \pmod{4}$ and is $-1$ if $n\equiv 2,3\pmod{4}$. Therefore, when $n\equiv 0,1 \pmod{4}$ the coefficients agree and $b_\wedge$ is symmetric.

Now assume $n\equiv 2,3 \pmod{4}$. The expression
\[
b_\wedge\left(\sum_{I\subseteq [n]}\alpha_I v_I, \sum_{I\subseteq [n]}\alpha_I v_I \right) = \sum_{I\subseteq [n]} (-1)^{(\Sigma I)-|I|} \alpha_I \alpha_{I^c}
\]
will contain the terms $(-1)^{(\Sigma I)-|I|} \alpha_I \alpha_{I^c}$ and $(-1)^{(\Sigma I^c)-|I^c|} \alpha_{I^c}\alpha_I$ in pairs. By the computation above, these are negatives of one another and hence cancel. Thus we see that $b_\wedge$ is alternating in this case.

\noindent\ref{b_wedge_properties_ii}: Pick any order on the basis $\{v_I \mid I\subseteq [n]\}$ of $\wedge\cV$. Once again using the formula of \Cref{b_wedge_formula}, we see that the bilinear form $b_\wedge$ will be given by $b(x,y)=x^T B y$ where we view $x$ and $y$ as column vectors and where $B=[\beta_{IJ}]$ is the matrix with entries
\[
\beta_{IJ} = \begin{cases} 0 & J\neq I^c \\ (-1)^{\Sigma I -|I|} & J=I^c. \end{cases}
\]
This matrix $B$ is invertible, so $b_\wedge$ is regular.

\noindent\ref{b_wedge_properties_iii}: To verify that $\tau = \Phi\circ\underline{\sigma}\circ \Phi^{-1}$, we only need to show that $\tau$ fixes the elements $\ell_{v_i}, d_{v_i^*} \in \cEnd_{\cO}(\wedge \cN)$ since the defining property of $\underline{\sigma}$ is that it acts as the identity on $\HH(\cV) \subseteq \Cl(\HH(\cV),q_{2n})$. First, we see that
\[
b(x,\ell_{v_i}(y)) = \pi(\overline{x}\wedge (v_i \wedge y)) = \pi(\overline{(v_i\wedge x)}\wedge y) = b(\ell_{v_i}(x),y)
\]
and so $\tau(\ell_{v_i})=\ell_{v_i}$ as desired. Second, to verify that $\tau(d_{v_i^*})=d_{v_i^*}$ it suffices to only consider basis elements of $\wedge\cV$. Let $J,K\subseteq [n]$. If $J\cup K \neq [n]$ or $J\cap K \neq \{i\}$, then it is clear that both
\[
b_\wedge(v_J,d_{v_i^*}(v_K))=0 \text{ and } b_\wedge(d_{v_i^*}(v_J),v_K)=0.
\]
Now assume that $J\cup K = [n]$ and $J\cap K = \{i\}$. Let $J=\{j_1,\ldots,j_p=i,\ldots,j_r\}$ and $K=\{k_1,\ldots,k_q=i,\ldots,k_s\}$. Then
\begin{align*}
\overline{v_J}\wedge d_{v_i^*}(v_K) &= \overline{v_J}\wedge (-1)^{q+1}v_{J^c} = (-1)^{q+1}(-1)^{\Sigma J - |J|} v_{[n]}, \text{ and}\\
\overline{d_{v_i^*}(v_J)}\wedge v_K &= (-1)^{p+1}\overline{v_{K^c}}\wedge v_K = (-1)^{p+1}(-1)^{\Sigma K^c -|K^c|} v_{[n]}.
\end{align*}
Due to the assumptions on $J$ and $K$, we have that $\Sigma K^c = \Sigma J - i$ and $|K^c| = |J|-1$. Furthermore, by counting the number of elements of $[n]$ less than $i$ in $J$ and $K$, we get that $i-1=(p-1)+(q-1)$. Therefore,
\begin{align*}
(-1)^{p+1}(-1)^{\Sigma K^c -|K^c|} &= (-1)^{p+1}(-1)^{\Sigma J - |J| - (i - 1)} = (-1)^{(p+1)-(i-1)}(-1)^{\Sigma J - |J|}\\
&= (-1)^{q+1}(-1)^{\Sigma J - |J|}
\end{align*}
and hence the above two expressions are equal. In summary, $b_\wedge(v_J,d_{v_i^*}(v_K)) = b_\wedge(d_{v_i^*}(v_J),v_K)$ for all $J,K\subseteq [n]$, and so $\tau(d_{v_i^*})=d_{v_i^*}$. This finishes the proof.
\end{proof}

Next, let $q_\wedge \colon \wedge \cV \to \cO$ be the quadratic form defined by
\begin{equation}\label{eq_q_wedge}
q_\wedge \left(\sum_{I\subseteq [n]} \alpha_I v_I\right) = \sum_{I\subseteq [n] \atop 1\in I} (-1)^{\Sigma I -|I|} \alpha_I \alpha_{I^c}
\end{equation}
for coefficients $a_I \in \cO$. Note that $q_\wedge$ is a hyperbolic quadratic form. This can be seen, for example, with respect to the basis of $\wedge\cV$ consisting of elements $v_I$ if $1\notin I$ and $(-1)^{\Sigma I -|I|}v_I$ if $1\in I$.
 
\begin{lem}\label{q_wedge_polar}
Assume $n\equiv 0,1\pmod{4}$ or $2=0 \in \cO$. Then, the bilinear form $b_\wedge$ is the polar of the quadratic form $q_\wedge$ of \eqref{eq_q_wedge}.
\end{lem}
\begin{proof}
We show that $b_\wedge$ is the polar of $q_\wedge$ by computation. For two generic elements $\sum_{I\subseteq [n]}\alpha_I v_I$ and $\sum_{I\subseteq [n]}\beta_I v_I$ in $\wedge\cV$, we have
\begin{align*}
&q_\wedge\left( \sum_{I\subseteq [n]}\alpha_I v_I + \sum_{I\subseteq [n]}\beta_I v_I \right) - q_\wedge \left(\sum_{I\subseteq [n]}\alpha_I v_I\right) - q_\wedge \left(\sum_{I\subseteq [n]}\beta_I v_I\right) \\
=& \sum_{I\subseteq [n] \atop 1\in I} (-1)^{\Sigma I -|I|} (\alpha_I+\beta_I)(\alpha_{I^c}+\beta_{I^c}) - \sum_{I\subseteq [n] \atop 1\in I} (-1)^{\Sigma I -|I|} \alpha_I \alpha_{I^c} - \sum_{I\subseteq [n] \atop 1\in I} (-1)^{\Sigma I -|I|} \beta_I \beta_{I^c} \\
=& \sum_{I\subseteq [n] \atop 1\in I} (-1)^{\Sigma I -|I|}(\alpha_I\beta_{I^c} + \alpha_{I^c}\beta_I) = \sum_{I\subseteq [n] \atop 1\in I}(-1)^{\Sigma I -|I|} \alpha_I\beta_{I^c} + \sum_{I\subseteq [n] \atop 1\notin I} (-1)^{\Sigma I^c -|I^c|}\alpha_I\beta_{I^c} \\
=& \sum_{I\subseteq [n]} (-1)^{\Sigma I -|I|}\alpha_I\beta_{I^c}
\end{align*}
where the last equality uses the fact from the proof of \Cref{b_wedge_properties}\ref{b_wedge_properties_i} that $(-1)^{\Sigma I -|I|}=(-1)^{\Sigma I^c -|I^c|}$ under our assumptions. This final expression is equal to
\[
b_\wedge(\sum_{I\subseteq [n]}\alpha_I v_I, \sum_{I\subseteq [n]}\beta_I v_I)
\]
by \Cref{b_wedge_formula}. Therefore $b_\wedge$ is the polar of $q_\wedge$.
\end{proof}
\end{ex}

The computations in the above example describe any Clifford algebra locally, and so they give us the following analogue of \cite[8.12]{KMRT}.
\begin{lem}\label{even_invol_type}
Let $(\cA,\sigma,f)$ be a quadratic triple of degree $2n$. Then, $\Cl(\cA,\sigma,f)$ is locally a Cartesian product of two Azumaya algebras and we have the following.
\begin{enumerate}[label={\rm(\roman*)}]
\item \label{even_invol_type_i} If $n$ is odd, the canonical involution $\underline{\sigma}$ acts non-trivially on the center of $\Cl(\cA,\sigma,f)$.
\item \label{even_invol_type_ii} If $n \equiv 0 \pmod{4}$ then $\underline{\sigma}$ is orthogonal.
\item \label{even_invol_type_iii} If $n\equiv 2 \pmod{4}$ then $\underline{\sigma}$ is symplectic.
\end{enumerate}
\end{lem}
\begin{proof}
Locally we will be in the case of \Cref{split_Clifford} where it is clear that
\[
\Cl_0(\HH(\cV),q_{2n}) \iso \cEnd_{\cO}(\wedge_0 \cV)\times \cEnd_{\cO}(\wedge_1 \cV)
\]
is a Cartesian product of two Azumaya algebras. By \Cref{b_wedge_properties}\ref{b_wedge_properties_iii} we may prove the claims for the adjoint involution $\tau$ of $b_\wedge$ restricted to the even Clifford algebra.

\noindent\ref{even_invol_type_i}: When $n$ is odd, if the basis element $v_I$ is in $\wedge_0 \cV$, then $v_{I^c} \in \wedge_1 \cV$ and vice-versa. Let $\varphi = (0,\Id) \in \cEnd_{\cO}(\wedge_0 \cV)\times \cEnd_{\cO}(\wedge_1 \cV)$. Then for some $I\subseteq [n]$ with $|I|$ even we have $b_\wedge(v_I,v_{I^c}) = \pm 1$ and
\[
b_\wedge(v_I,v_{I^c}) = b_\wedge(v_I,\varphi(v_{I^c})) = b_\wedge((\tau(\varphi))(v_I),v_{I^c}).
\]
Therefore, $\tau(\varphi)\neq \varphi$, demonstrating that $\tau$ acts non-trivially on the center.

\noindent\ref{even_invol_type_ii}, \ref{even_invol_type_iii}: When $n$ is even, the bilinear form $b_\wedge$ restricts to a regular bilinear form on both $\wedge_0 \cV$ and $\wedge_1\cV$. These forms are symmetric or alternating when $b_\wedge$ is, and so these claims follow from \Cref{b_wedge_properties}\ref{b_wedge_properties_i}.
\end{proof}

\subsection{Algebraic Groups}
In our setting over a scheme, algebraic groups will be sheaves of groups on $\Sch_S$. In particular, we are interested in groups of type $D$. We recall the definitions of some key groups from \cite{CF}.
\begin{defn}\label{defn_GL_PGL}
Let $\cB$ be an $\cO$--algebra. The \emph{general linear group} of $\cB$ is the subsheaf of invertible elements
\begin{align*}
\GL_{\cB} \colon \Sch_S &\to \Grp \\
T &\mapsto \cB(T)^\times .
\end{align*}

The \emph{projective general linear group} is $\PGL_{\cB} = \cAut_{\cO\mathrm{\mdash alg}}(\cB)$. If $\cM$ is an $\cO$--module, we set $\GL_{\cM} = \GL_{\cEnd_{\cO}(\cM)} = \cAut_{\cO\mathrm{\mdash mod}}(\cM)$ and $\PGL_{\cM} = \PGL_{\cEnd_{\cO}(\cM)}$. We set $\GL_n = \GL_{\cO^n} \cong \GL_{\Mat_n(\cO)}$ and $\PGL_n$ similarly.
\end{defn}

\begin{defn}\label{defn_runity}
The group of $n^{\textnormal{th}}$--roots of unity is
\begin{align*}
\runity_n \colon \Sch_S &\to \Grp \\
T &\mapsto \{x \in \cO(T) \mid x^n = 1\}.
\end{align*}
\end{defn}

Our primary focus will be on groups of type $D$, which are those related to quadratic forms and quadratic triples.
\begin{defn}\label{defn_O_q}
Let $(\cM,q)$ be a regular quadratic module. The \emph{orthogonal group} is the group of automorphisms which fix $q$,
\begin{align*}
\bO_q \colon \Sch_S &\to \Grp \\
T &\mapsto \{\varphi \in \GL_{\cM}(T) \mid q|_T \circ \varphi = q|_T\}.
\end{align*}
When $(\cM,q)=(\HH(\cV),q_{2n})$ is the even rank hyperbolic quadratic form of \Cref{split_Clifford}, we write $\bO_{q_{2n}} = \bO_{2n}$.
\end{defn}

\begin{defn}\label{defn_O_PGO}
Let $(\cA,\sigma,f)$ be a quadratic triple. The \emph{orthogonal group} of $(\cA,\sigma,f)$ is
\begin{align*}
\bO_{(\cA,\sigma,f)} \colon \Sch_S &\to \Grp \\
T &\mapsto \{ a \in \cA(T) \mid a\sigma(a)=1,\; f|_T \circ \Inn(a) = f|_T \}
\end{align*}
where $\Inn(a) \colon \cA|_T \to \cA|_T$ is the inner automorphism defined by the section $a \in \cA(T)$.

The \emph{projective orthogonal group} is the group of quadratic triple automorphisms
\[
\PGO_{(\cA,\sigma,f)} = \cAut(\cA,\sigma,f).
\]
If $(\cM,q)$ is a regular quadratic module, we set $\PGO_{q} = \PGO_{(\cEnd_{\cO}(\cM),\sigma_q,f_q)}$. For $(\HH(\cV),q_{2n})$ of \Cref{split_Clifford}, we write $\PGO_{2n}$.
\end{defn}

By \cite[4.4.0.44]{CF}, for a regular quadratic module $(\cM,q)$, there is a canonical isomorphism $\bO_q \cong \bO_{(\cEnd_{\cO}(\cM),\sigma_q,f_q)}$. For a quadratic triple $(\cA,\sigma,f)$, the canonical projection
\begin{align*}
\pi_{\bO}\colon \bO_{(\cA,\sigma,f)} &\surj \PGO_{(\cA,\sigma,f)} \\
a &\mapsto \Inn(a)
\end{align*}
is a surjective map of sheaves with kernel isomorphic to $\runity_2$.

Let $(\cM,q)$ be a regular quadratic module. We have group homomorphisms
\begin{align}
\bC \colon \bO_q &\to \cAut(\Cl(\cM,q),\underline{\sigma}), \text{ and} \label{eq_Clifford_actions} \\
\bC' \colon \PGO_q &\to \cAut(\Cl_0(\cM,q),\underline{\sigma}_0). \nonumber
\end{align}
For $B\in \bO_q$, its image $\bC(B)$ is uniquely defined by $(\bC(B))(m) = B(m)$ for all $m\in \cM$. For $\psi \in \PGO_q$, its image $\bC'(\psi)$ is uniquely defined by $(\bC'(\psi))(m_1\otimes m_2) = (\varphi_{b_q}^{-1}\circ \psi \circ \varphi_{b_q})(m_1\otimes m_2)$ for $m_1,m_2 \in \cM$ using the isomorphism $\varphi_{b_q}$ of \eqref{eq_bilinear_iso}. Of course, for $\psi \in \PGO_q$, its image under $\bC'$ is also defined by the functoriality of the Clifford algebra, namely $\bC'(\psi) = \Cl(\psi)$ as in \eqref{eq_Clifford_functorial}. Note that the image of $\bC$ stabilizes the even Clifford algebra, and so by abuse of notation we will also discuss the homomorphism
\[
\bC \colon \bO_q \to \cAut(\Cl_0(\cM,q),\underline{\sigma}_0).
\]

\begin{lem}\label{action_on_Clifford}
The diagram
\[
\begin{tikzcd}
\bO_q \arrow{d}{\pi_{\bO}} \arrow{dr}{\bC} & \\
\PGO_q \arrow{r}{\bC'} & \bAut(\Cl_0(\cM,q),\underline{\sigma}_0)
\end{tikzcd}
\]
commutes.
\end{lem}
\begin{proof}
Let $B \in \bO_q$. By definition, the morphism $B$ satisfies $\sigma(B)=B^{-1}$, equivalently $B=\sigma(B^{-1})$. For $m_1\otimes m_2 \in \cM\otimes_{\cO}\cM$ we can compute
\begin{align*}
(\varphi_{b_q}^{-1}\circ \pi_{\bO}(B) \circ \varphi_{b_q})(m_1\otimes m_2) &= (\varphi_{b_q}^{-1} \circ \pi_{\bO}(B))(b_q(m_1,\und)m_2) \\
&= \varphi_{b_q}^{-1}(B\circ b_q(m_1,\und)m_2 \circ B^{-1}) \\
&= \varphi_{b_q}^{-1}(b_q(m_1,B^{-1}(\und))\cdot B(m_2)) \\
&= \varphi_{b_q}^{-1}(b_1(B(m_1),\und)\cdot B(m_2))\\
&= B(m_1)\otimes B(m_2).
\end{align*}
Therefore, the same thing holds in the Clifford algebra, justifying that $\bC = \bC' \circ \pi_{\bO}$.
\end{proof}
Since both of the homomorphisms in \eqref{eq_Clifford_actions} factor through $\cAut(\Cl_0(\cM,q),\underline{\sigma}_0)$, twisted versions of these maps also exist. In particular, for a quadratic triple $(\cA,\sigma,f)$ there is a commutative diagram
\begin{equation}\label{eq_twisted_Clifford_actions}
\begin{tikzcd}
\bO_{(\cA,\sigma,f)} \arrow{d}{\pi_{\bO}} \arrow{dr}{\bC} &  \\
\PGO_{(\cA,\sigma,f)} \arrow{r}{\bC'} & \cAut(\Cl(\cA,\sigma,f),\underline{\sigma})
\end{tikzcd}
\end{equation}
where the maps locally agree with the ones in \eqref{eq_Clifford_actions}.

\subsection{Twisting via Torsors}\label{cohomology_twisting}
Let $\bG \colon \Sch_S \to \Grp$ be any sheaf of groups. A \emph{$\bG$--torsor} is a sheaf of sets $\cP \colon \Sch_S \to \Grp$ with a right $\bG$--action, i.e., a natural transformation $\cP\times \bG \to \cP$ such that
\begin{enumerate}[label={\rm(\roman*)}]
\item \label{torsor_defn_i} for all $T \in \Sch_S$, the map $\cP(T)\times \bG(T) \to \cP(T)$ defines a simply transitive right $\bG(T)$--action on $\cP(T)$, and
\item \label{torsor_defn_ii} there exists a cover $\{T_i \to S\}_{i\in I}$ over which $\cP(T_i)\neq \O$ for all $i\in I$.
\end{enumerate}
As \ref{torsor_defn_ii} above hints, a torsor $\cP$ may have $\cP(T) = \O$ for some $T\in \Sch_S$ and in this case the map $\O \times \bG(T) \to \O$ vacuously gives a simply transitive action of $\bG(T)$ on $\O$. A morphism of $\bG$--torsors is a natural transformation of sheaves which is equivariant with respect to the right $\bG$--actions. Any morphism of torsors must be an isomorphism. The \emph{trivial torsor} is $\bG$ itself viewed as a sheaf of sets with right action coming from right multiplication.

We will use contracted products with respect to torsors. If $\cP$ is a $\bG$--torsor and $\cF$ is another sheaf of sets equipped with a left action of $\bG$, then the contracted product as defined in \cite[2.2.2.9]{CF} is
\[
\cP \wedge^{\bG} \cF = ((\cP \times \cF)/\sim)^\sharp,
\]
where the $\sharp$ denotes sheafification of the presheaf $\cP \times \cF$ modulo the equivalence relation $(p\cdot g,f)\sim (p,g\cdot f)$ for all appropriate sections $g\in \bG$, $p\in \cP$, and $f \in \cF$. We call $\cP \wedge^{\bG} \cF$ the \emph{twist of $\cF$ by $\cP$}.

Two sheaves $\cF,\cG$ on $\Sch_S$ are called \emph{twisted forms} of one another if there exists a cover $\{T_i \to S\}_{i\in I}$ such that $\cF|_{T_i} \cong \cG|_{T_i}$ for all $i\in I$. In this case, the sheaf of internal isomorphisms
\begin{align*}
\cIsom(\cF,\cG) \colon \Sch_S &\to \Sets \\
T &\mapsto \Isom(\cF|_T,\cG|_T)
\end{align*}
is an $\cAut(\cF)$--torsor. Furthermore, $\cIsom(\cF,\cG)\wedge^{\cAut(\cF)} \cF \cong \cG$. We will primarily use this in the context of quadratic triples. For example, if $(\cA,\sigma,f)$ is a quadratic triple of constant rank $2n$, then the sheaf of quadratic triple isomorphisms
\[
\cP = \cIsom((\cEnd_{\cO}(\HH(\cV)),\sigma_{q_{2n}},f_{q_{2n}}),(\cA,\sigma,f))
\]
is a $\PGO_{2n}$--torsor and the twist of $(\cEnd_{\cO}(\HH(\cV)),\sigma_{q_{2n}},f_{q_{2n}})$ by $\cP$ is $(\cA,\sigma,f)$.

\section{The Canonical Semi-Trace on Clifford Algebras}\label{Clifford_triples}
Here we extend the construction of the canonical semi-trace on the Clifford algebra of a quadratic triple from \cite{DQ21} where it was first defined over a field of characteristic $2$. Their construction relies on the existence of an element in the Azumaya algebra with trace $1$, however there may fail to be such a global section when working over a scheme. Therefore, we present an alternate construction which shares many of the same properties and which agrees with their construction over fields.

\subsection{When $\cA$ is of Degree $\geq 8$}\label{A_degree_8}
We recall the definition from \cite{DQ21}.
\begin{defn}[{\cite[3.3]{DQ21}}]
Let $(A,\sigma,f)$ be a central simple algebra of degree $2n\geq 8$ over a field $\FF$ with quadratic pair $(\sigma,f)$. Assume that $n$ is even and, in the case that $\Char(\FF)\neq 2$, that $n\equiv 0 \pmod{4}$. Given $\lambda \in A$ with $\Trd_A(\lambda)=1$, the semi trace
\begin{align*}
\underline{f} \colon \Sym(\mathrm{Cl}(A,\sigma,f),\underline{\sigma}) &\to \FF \\
s &\mapsto \Trd_{\mathrm{Cl}(A,\sigma,f)}(c(\lambda)s),
\end{align*}
where $c \colon A \to \mathrm{Cl}(A,\sigma,f)$ is the canonical mapping, does not depend on $\lambda$. It is called the \emph{canonical semi-trace} on $(\mathrm{Cl}(A,\sigma,f),\underline{\sigma})$.
\end{defn}
The assumptions on $n$ are to ensure that the involution $\underline{\sigma}$ is orthogonal, see \Cref{even_invol_type}. This definition works because, on account of the trace form $(a,b)\mapsto \Trd_A(ab)$ on $A$ being a regular symmetric bilinear form, there will be an element in $A$ with non-zero trace, and then since $\FF$ is a field the element can be scaled to have trace $1$ as required. In fact, this property is also preserved over affine schemes.
\begin{lem}\label{affine_trace_1}
Let $\cA$ be an Azumaya $\cO$--algebra. For all $U \in \Aff_S$, there exists a section $a\in \cA(U)$ such that $\Trd_{\cA}(a) = 1 \in \cO(U)$.
\end{lem}
\begin{proof}
Since $U$ is an affine scheme, the algebra $\cA(U)$ is an Azumaya algebra over the ring $\cO(U)$. The result is then \cite[11.1.6]{Ford}.
\end{proof}

However, if $S$ is not affine then we may fail to have a global section with trace $1$. This is demonstrated in the following example.
\begin{ex}\label{no_trace_1}
We recall the example constructed in \cite[7.1]{GNR}. There, the base scheme is an ordinary elliptic curve $E$ defined over a field $\FF$ of characteristic $2$. The authors construct a quaternion $\cO$--algebra $\cQ$, which by \cite[7.2(i)]{GNR} belongs to some quadratic triple $(\cQ,\sigma,f)$. However, by \cite[7.2(ii)]{GNR}, the global sections of this algebra are $\cQ(E) \cong \FF$. Now, let $\{T_i \to E\}_{i\in I}$ be any cover which splits $\cQ$, i.e., $\cQ|_{T_i} \cong \Mat_2(\cO|_{T_i})$ for all $i\in I$. The restrictions maps $\cQ(E) \to \cQ(T_i)$ must then be of the form
\begin{align*}
\FF &\to \Mat_2(\cO(T_i)) \\
c &\mapsto \begin{bmatrix} c & 0 \\ 0 & c \end{bmatrix}.
\end{align*}
In particular, for $c\in \cQ(E)$ we have $\Trd_{\cQ}(c|_{T_i})=2c = 0$ for all $i\in I$ since $\FF$ is characteristic $2$. But, this means that $\Trd_{\cQ}(c)=0$ globally, and so there does not exist a section in $\cQ(E)$ with trace $1$.
\end{ex}

Thus, to adapt the definition from \cite{DQ21} we must define our desired semi-trace differently. The key ingredient is the following lemma whose proof is adapted from the proof of \cite[3.2]{DQ21}. Recall that $\slLie_{\cA}$ is the kernel of the trace $\Trd_{\cA} \colon \cA \to \cO$.
\begin{lem}\label{sl_into_Alt}
Let $(\cA,\sigma,f)$ be a quadratic triple of degree $2n$ with $n\geq 3$. Then the canonical mapping $c\colon \cA \to \Cl(\cA,\sigma,f)$ of \eqref{eq_c_A_inclusion} restricts to
\[
c \colon \slLie_{\cA} \to \cAlt_{(\Cl(\cA,\sigma,f),\underline{\sigma})}.
\]
\end{lem}
\begin{proof}
It is sufficient to show that $c$ maps $\slLie_{\cA}$ into $\cAlt_{(\Cl(\cA,\sigma,f),\underline{\sigma})}$ locally, and so we may assume that $(\cA,\sigma,f)$ is the adjoint quadratic triple to $(\HH(\cV),q_{2n})$ from \Cref{split_Clifford}. Under the isomorphism $\varphi_{b_{q_{2n}}} \colon \HH(\cV)\otimes_{\cO} \HH(\cV) \iso \cEnd_{\cO}(\HH(\cV))$ of \eqref{eq_bilinear_iso}, the submodule $\slLie_{\cEnd_{\cO}(\HH(\cV))}$ has a basis consisting of elements
\begin{align*}
&v_i\otimes v_i, v_i^*\otimes v_i^*, & &\text{ for } i\in [n],\\
&v_i\otimes v_j,v_i\otimes v_j^*, v_j\otimes v_i^*, v_j^*\otimes v_i^*, & &\text{ for } i<j \in [n],\\
&v_i\otimes v_i^* - v_{i+1}\otimes v_{i+1}^*, v_n\otimes v_n^* - v_n^*\otimes v_n, v_{i+1}^*\otimes v_{i+1} - v_i^*\otimes v_i & &\text{ for } i\in [n-1].
\end{align*}
We show that these are all mapped by $c$ to alternating elements. For the first row, $c(v_i\otimes v_i)=0 $ and $c(v_i^*\otimes v_i^*) = 0$, which is of course alternating. 

For basis elements in the second row, let $x$ be the first tensor factor and $y$ be the second. Then $c(x\otimes y) = xy$ and $yx = -xy$. Furthermore, since $n\geq 3$ there exists $v_k$ with $k\neq i,j$, and therefore $xy$ commutes with both $v_k$ and $v_k^*$. Now, we can compute
\begin{align*}
(xyv_k v_k^*) - \underline{\sigma}(xy v_k v_k^*) &= xyv_k v_k^* - v_k^* v_k yx \\
&= xy v_k v_k^* + xy v_k^*v_k \\
&= xy v_k v_k^* + xy(1-v_k v_k^*) \\
&= xy
\end{align*}
which shows that $c(x\otimes y)$ is an alternating element.

For the last row, $c(v_n\otimes v_n^* - v_n^*\otimes v_n) = v_n v_n^* - v_n^* v_n = v_nv_n^* - \underline{\sigma}(v_n v_n^*)$ is clearly alternating. For the other basis elements in the last row we use a similar trick as in the second row. These elements are all mapped by $c$ to something of the form $x x^* - y y^*$ (considering $(x^*)^* = x$) and we may once again choose $k\in [n]$ such that $x,x^*,y,y^*$ all commute with $v_k$ and $v_k^*$. Then we compute
\begin{align*}
&(xx^*-yy^*)v_kv_k^* - \underline{\sigma}((xx^*-yy^*)v_kv_k^*)\\
=& (xx^*-yy^*)v_kv_k^* - v_k^*v_k(x^* x - y^* y) \\
=& (xx^*-yy^*)v_kv_k^* - (1-v_kv_k^*)(1-xx^* - 1 + yy^*) \\
=& (xx^*-yy^*)v_kv_k^* + (xx^* - yy^*) -v_kv_k^*(xx^*-yy^*) \\
=& (xx^*-yy^*)v_kv_k^* + (xx^* - yy^*) -(xx^*-yy^*)v_kv_k^* \\
=& xx^* - yy^*
\end{align*}
and therefore $c(x\otimes x^* - y\otimes y^*)$ is an alternating element. This finishes the proof.
\end{proof}

\begin{rem}\label{degree_4_trace_remark}
Unfortunately, \Cref{sl_into_Alt} is not true if $(\cA,\sigma,f)$ is of degree $4$. For example, consider the trace zero element $v_1\otimes v_2 \in \HH(\cV)\otimes_{\cO} \HH(\cV)$. We have $c(v_1\otimes v_2) = v_1v_2$, which is then mapped by the isomorphism $\Phi_0$ to the endomorphism $\ell_{v_1}\circ \ell_{v_2}$. Choosing the ordered bases $\{1,v_1\wedge v_2\}$ of $\wedge_0 \cV$ and $\{v_1,v_2\}$ of $\wedge_1 \cV$, we see that $\ell_{v_1}\circ \ell_{v_2}$ has matrix
\[
\left( \begin{bmatrix} 0 & 0 \\ 1 & 0 \end{bmatrix},\begin{bmatrix} 0 & 0 \\ 0 & 0 \end{bmatrix}\right).
\]
However, by \Cref{b_wedge_properties}\ref{b_wedge_properties_iii} the involution $\underline{\sigma}$ appears here as
\[
\left( \begin{bmatrix} a & b \\ c & d \end{bmatrix},\begin{bmatrix} a' & b' \\ c' & d' \end{bmatrix}\right) \mapsto \left( \begin{bmatrix} d & b \\ c & a \end{bmatrix},\begin{bmatrix} d' & b' \\ c' & a' \end{bmatrix}\right)
\]
and therefore the alternating elements are of the form
\[
\left( \begin{bmatrix} a-d & 0 \\ 0 & d-a \end{bmatrix},\begin{bmatrix} a'-d' & 0 \\ 0 & d'-a' \end{bmatrix}\right).
\]
Thus, $c(v_1\otimes v_2)$ is not an alternating element despite being trace zero.
\end{rem}

\Cref{sl_into_Alt} means that when we consider the commutative diagram of $\cO$--modules with exact rows
\begin{equation}\label{eq_rho_diagram}
\begin{tikzcd}[column sep=3ex]
0 \arrow{r} & \slLie_{\cA} \arrow[hookrightarrow]{r} \arrow{d}{c} & \cA \arrow{r}{\Trd_{\cA}} \arrow{d}{c} & \cO \arrow{d}{\exists ! \rho} \arrow{r} & 0 \\
0 \arrow{r} & \cAlt_{(\Cl(\cA,\sigma,f),\underline{\sigma})} \arrow[hookrightarrow]{r} & \Cl(\cA,\sigma,f) \arrow{r}{\pi} & \Cl(\cA,\sigma,f)/\cAlt_{(\Cl(\cA,\sigma,f),\underline{\sigma})} \arrow{r} & 0
\end{tikzcd}
\end{equation}
where $\pi$ is the canonical projection, then there exists a unique well defined morphism $\rho$ induced by $c$. We recall from \cite[1.6.1]{GNR} that we also have another diagram with exact rows, where we occasionally abbreviate $\Cl(\cA,\sigma,f)=\Cl(\cA)$,
\[
\begin{tikzcd}[column sep=3ex]
0 \arrow{r} & \cSkew_{(\Cl(\cA),\underline{\sigma})} \arrow[hookrightarrow]{r} \arrow{d} & \Cl(\cA,\sigma,f) \arrow{r}{\Id+\underline{\sigma}} \arrow{d}{\pi} & \cSymd_{(\Cl(\cA),\underline{\sigma})} \arrow[equals]{d} \arrow{r} & 0 \\
0 \arrow{r} & \cSkew_{(\Cl(\cA),\underline{\sigma})}/\cAlt_{(\Cl(\cA),\underline{\sigma})} \arrow[hookrightarrow]{r} & \Cl(\cA,\sigma,f)/\cAlt_{(\Cl(\cA),\underline{\sigma})} \arrow{r}{\xi} & \cSymd_{(\Cl(\cA),\underline{\sigma})} \arrow{r} & 0.
\end{tikzcd}
\]
Here, $\xi$ is the morphism induced by $\Id+\underline{\sigma}$.
\begin{lem}\label{rho_xi_compatible}
Let $\rho$ and $\xi$ be as in the diagrams above. The map
\[
\cO \xrightarrow{\rho} \Cl(\cA,\sigma,f)/\cAlt_{(\Cl(\cA,\sigma,f),\underline{\sigma})} \xrightarrow{\xi} \cSymd_{(\Cl(\cA,\sigma,f),\underline{\sigma})}
\]
sends $\alpha \mapsto \alpha\cdot 1_{\Cl(\cA,\sigma,f)}$.
\end{lem}
\begin{proof}
Let $\alpha\in \cO$. Since $\Trd_{\cA}$ is surjective, there exists a cover $\{T_i \to S\}_{i\in I}$ and elements $a_i \in \cA(T_i)$ with $\Trd_{\cA}(a_i) = \alpha|_{T_i}$ for all $i\in I$. For each of these, we have
\begin{align*}
((\Id+\underline{\sigma})\circ c)(a_i) &= c(a_i) + \underline{\sigma}(a_i) = c(a_i +\sigma(a_i)) \\
&= f(a_i+\sigma(a_i))\cdot 1 = \Trd_{\cA}(a_i)\cdot 1 = \alpha|_{T_i}\cdot 1.
\end{align*}
Since the diagrams above commute, 
\[
(\xi \circ \rho)(\alpha_{T_i}) = (\xi \circ \pi\circ c)(a_i) = ((\Id +\underline{\sigma})\circ c)(a_i) = \alpha|_{T_i}\cdot 1.
\]
Therefore $(\xi\circ \rho)(\alpha) = \alpha\cdot 1$ globally as claimed.
\end{proof}
In particular, \Cref{rho_xi_compatible} shows $\xi(\rho(1)) = 1_{\Cl(\cA,\sigma,f)} \in \cSymd_{(\Cl(\cA,\sigma,f),\underline{\sigma})}(S)$, which means that the section $\rho(1) \in \Cl(\cA,\sigma,f)/\cAlt_{(\Cl(\cA,\sigma,f),\underline{\sigma})}$ may be used to define a quadratic triple.
\begin{defn}\label{defn_canonical_triple}
Let $(\cA,\sigma,f)$ be a quadratic triple of degree $2n$ with $n\geq 4$ such that $(\Cl(\cA,\sigma,f),\underline{\sigma})$ has an orthogonal involution. Let $\rho\colon \cO \to \Cl(\cA,\sigma,f)/\cAlt_{(\Cl(\cA,\sigma,f),\underline{\sigma})}$ be the map of \eqref{eq_rho_diagram}. We define a semi-trace
\begin{align*}
\underline{f} \colon \cSym_{(\Cl(\cA,\sigma,f),\underline{\sigma})} &\to \cO \\
s &\mapsto \Trd_{\Cl(\cA,\sigma,f)}(\rho(1)\cdot s)
\end{align*}
as in \Cref{quad_forms_and_triples}. We call $\underline{f}$ the \emph{canonical semi-trace} on $(\Cl(\cA,\sigma,f),\underline{\sigma})$ and we call the triple $(\Cl(\cA,\sigma,f),\underline{\sigma},\underline{f})$ the \emph{canonical Clifford algebra with semi-trace}.
\end{defn}
\begin{rem}
The requirement that $\underline{\sigma}$ be orthogonal will be satisfied when $n\equiv 0 \pmod{4}$ by \Cref{even_invol_type}\ref{even_invol_type_ii}. However, it may also be satisfied when $n\equiv 2\pmod{4}$ if $2=0\in \cO$. By \Cref{even_invol_type}\ref{even_invol_type_iii}, $\underline{\sigma}$ will be symplectic in this case, but because $2=0$ it will be orthogonal as well.
\end{rem}

\begin{thm}\label{canonical_f_local}
Let $(\cA,\sigma,f)$ be a quadratic triple over $S$ and let $(\Cl(\cA,\sigma,f),\underline{\sigma},\underline{f})$ be the canonical Clifford algebra with semi-trace. Then, for all $T\in \Sch_S$ for which there exists a section $a\in \cA(T)$ with $\Trd_{\cA}(a)=1 \in \cO(T)$, we have
\[
\underline{f}|_T = \Trd_{\Cl(\cA,\sigma,f)}(c(a) \und )
\]
and this does not depend on the choice of $a \in \cA(T)$. In particular, this definition agrees with the canonical semi-trace of \cite[3.3]{DQ21} when $S = \Spec(\FF)$ for a field $\FF$.
\end{thm}
\begin{proof}
By definition of a semi-trace with respect to the element
\[
\rho(1) \in (\Cl(\cA,\sigma,f)/\cAlt_{(\Cl(\cA,\sigma,f),\underline{\sigma})})(S),
\]
whenever there exists $T\in \Sch_S$ with an element $x\in \Cl(\cA,\sigma,f)(T)$ such that $\pi(x) = \rho(1)$, then $\underline{f}|_T = \Trd_{\cA}(x\und)$. If there exists $a \in \cA(T)$ with $\Trd_{\Cl(\cA,\sigma,f)}(a) = 1$, then by commutativity of \eqref{eq_rho_diagram}, the section $c(a) \in \Cl(\cA,\sigma,f)(T)$ is such an element. This justifies the claim.
\end{proof}

The description of $\underline{f}$ given in \Cref{canonical_f_local} allows us to easily show that the definition of the canonical semi-trace is stable under base change.
\begin{cor}\label{canonical_f_stable_base_change}
Let $(\cA,\sigma,f)$ be a quadratic triple of degree $2n$ with $n\geq 4$ such that we have the canonical Clifford algebra with semi-trace $(\Cl(\cA,\sigma,f),\underline{\sigma},\underline{f})$. Let $T\in \Sch_S$ be any scheme. The restricted quadratic triple $(\cA|_T,\sigma|_T,f|_T)$ also satisfies the required assumptions and so there is a canonical $\cO|_T$--algebra with semi-trace $(\Cl(\cA|_T,\sigma|_t,f|_T),\underline{\sigma|_T},\underline{f|_T})$. Then,
\[
(\Cl(\cA|_T,\sigma|_t,f|_T),\underline{\sigma|_T},\underline{f|_T}) = (\Cl(\cA,\sigma,f),\underline{\sigma},\underline{f})|_T,
\]
i.e., the construction of the canonical semi-trace is stable under base change.
\end{cor}
\begin{proof}
The construction of the Clifford algebra and its canonical involution are stable under base change, so we only need to verify that $\underline{f|_T} = \underline{f}|_T$. Let $\{U_i \to T\}_{i\in I}$ be an affine cover of $T$. By \Cref{affine_trace_1}, since each $U_i$ is affine there will be sections $a_i\in \cA|_T(U_i) = \cA(U_i)$ with $\Trd_{\cA}(a_i) = 1$. Therefore by applying \Cref{canonical_f_local} we have that
\[
(\underline{f|_T})|_{U_i} = \Trd_{\Cl(\cA,\sigma,f)}(c(a_i) \und ) = \underline{f}|_{U_i} = (\underline{f}|_T)|_{U_i}.
\]
Hence, since they are equal locally on a cover we conclude that $\underline{f|_T} = \underline{f}|_T$ over $T$ as desired.
\end{proof}

The next proposition shows that, in the split case, the canonical semi-trace aligns with the quadratic form defined in \Cref{split_Clifford}.
\begin{prop}\label{canonical_Phi_correspondence}
Let $(\HH(\cV),q_{2n})$ be as in \Cref{split_Clifford} with $n\geq 4$ such that $(\Cl_0(\HH(\cV),q_{2n}),\underline{\sigma}_0)$ has an orthogonal involution. Then, the isomorphism $\Phi_0$ of \eqref{eq_Phi_not} extends to an isomorphism of algebras with semi-trace
\[
\Phi_0 \colon (\Cl_0(\HH(\cV),q_{2n}),\underline{\sigma}_0,\underline{f}) \iso (\cEnd_{\cO}(\wedge_0 \cV),\sigma_{\wedge 0},f_{\wedge 0})\times (\cEnd_{\cO}(\wedge_1 \cV),\sigma_{\wedge 1},f_{\wedge 1}),
\]
where $\underline{f}$ is the canonical semi-trace and $(\cEnd_{\cO}(\wedge_i \cV),\sigma_{\wedge i},f_{\wedge i})$ are the adjoint quadratic triples associated to the hyperbolic quadratic forms $q_{\wedge}|_{\wedge_i \cV}$.
\end{prop}
\begin{proof}
The assumption that $\underline{\sigma}_0$ be orthogonal means that $n \equiv 0 \pmod{4}$ or both $n\equiv 2 \pmod{4}$ and $2=0\in \cO$. Therefore, by \Cref{even_invol_type} and \Cref{q_wedge_polar}, the isomorphism $\Phi_0$ of \eqref{eq_Phi_not} extends to an isomorphism of algebras with involution
\[
\Phi_0 \colon (\Cl_0(\HH(\cV),q_{2n}),\underline{\sigma}_0) \iso (\cEnd_{\cO}(\wedge_0 \cV),\sigma_{\wedge, 0})\times (\cEnd_{\cO}(\wedge_1 \cV),\sigma_{\wedge,1}).
\]
Therefore, we only need to show that the semi-traces also correspond via this isomorphism. For computations, we choose the trace $1$ element $E_{2n,2n} \in \Mat_{2n}(\cO)$, which has image $c(E_{2n,2n})=v_1v_1^*$ in the Clifford algebra. Under $\Phi_0$ this maps to the endomorphism $\ell_{v_1}\circ d_{v_1^*}$. For $I\subseteq [n]$, we have
\[
(\ell_{v_1}\circ d_{v_1^*})(v_I) = \begin{cases} v_I & 1\in I \\ 0 & 1 \notin I. \end{cases}
\]
The semi-trace corresponding to $\underline{f}$ via $\Phi_0$ is $\Trd_{\cEnd_{\cO}(\wedge_0 \cV)\times \cEnd_{\cO}(\wedge_1 \cV)}((\ell_{v_1}\circ d_{v_1^*})\und)$.

Consider the isomorphism
\[
\varphi_{b_{\wedge,0}}\colon \wedge_0\cV \otimes_{\cO} \wedge_0 \cV \iso \cEnd_{\cO}(\wedge_0 \cV)
\]
of \eqref{eq_bilinear_iso}. Let $x = \displaystyle\sum_{I \subseteq [n] \atop |I| \text{ even}} a_I v_I$ be a generic element of $\wedge_0 \cV$. Then,
\[
\varphi_{b_{\wedge,0}}(x\otimes x) = \sum_{I,J\subseteq [n] \atop |I|,|J| \text{ even}} a_I a_J b_\wedge(v_I,\und)\cdot v_J.
\]
The trace on $\Cl_0(\HH(\cV),q_{2n})$ is the restriction of the reduced trace on $\Cl(\HH(\cV),q_{2n})$ where, using \Cref{b_wedge_formula}, we have
\[
\Trd_{\Cl(\HH(\cV),q_{2n})}(b_\wedge(v_I,\und)v_J) = \begin{cases} (-1)^{\Sigma I -|I|} & J= I^c \\ 0 & J\neq I^c. \end{cases}
\]
Furthermore,
\[
\ell_{v_1}\circ d_{v_1^*} \circ b_\wedge(v_I,\und)v_J = \begin{cases} b_\wedge(v_I,\und)\cdot v_J & 1\in J \\ 0 & 1 \notin J. \end{cases}
\]
Assembling these facts and using that $n$ is even, we compute that
\begin{align*}
(\Trd((\ell_{v_1}\circ d_{v_1^*})\und)\circ \varphi_{b_{\wedge,0}})(x\otimes x) &= \sum_{I \subseteq [n] \atop |I| \text{ even, } 1 \in I^c} (-1)^{\Sigma I - |I|}a_I a_{I^c} \\
&= \sum_{I \subseteq [n] \atop |I| \text{ even, } 1 \in I} (-1)^{\Sigma I^c - |I^c|}a_{I^c} a_{I} \\
&= \sum_{I \subseteq [n] \atop |I| \text{ even, } 1 \in I} (-1)^{\Sigma I - |I|}a_{I}a_{I^c}
\end{align*}
where in the last equality we us the fact that $(-1)^{\Sigma I - |I|}=(-1)^{\Sigma I^c - |I^c|}$ trivially if $2=0 \in \cO$ and they are equal by the argument in the proof of \Cref{b_wedge_properties}\ref{b_wedge_properties_i} if $n\equiv 0 \mod{4}$. However, this final expression is exactly $q_\wedge(x) = q_{\wedge 0}(x)$, and therefore by \cite[4.18]{GNR} the restriction of the canonical semi-trace to $\cEnd_{\cO}(\wedge_0 \cV)$ agrees with the semi-trace $f_{\wedge 0}$ from the quadratic triple adjoint to $q_{\wedge 0}$. The calculations for $\wedge_1 \cV$ and $q_{\wedge 1}$ are analogous and we leave them to the reader.
\end{proof}

This definition of the canonical semi-trace is also compatible with the functoriality of the Clifford algebra construction for isomorphisms of the underlying quadratic triples. Over fields this was shown in \cite[3.5]{DQ21}.
\begin{lem}\label{canonical_functorial}
Let $(\cA,\sigma,f)$ and $(\cA',\sigma',f')$ be two quadratic triples with canonical maps $c$ and $c'$ into their respective Clifford algebras. Let $\varphi \colon (\cA,\sigma,f)\iso (\cA',\sigma',f')$ be an isomorphism of quadratic triples. Then the morphism of \eqref{eq_Clifford_functorial} is also an isomorphism of algebras with semi-trace
\[
\Cl(\varphi) \colon (\Cl(\cA,\sigma,f),\underline{\sigma},\underline{f}) \to (\Cl(\cA',\sigma',f'),\underline{\sigma'},\underline{f'}).
\]
\end{lem}
\begin{proof}
Since $\Cl(\varphi)$ is already an isomorphism of algebras with involution, it only remains to check that $\underline{f'}\circ \Cl(\varphi) = \underline{f}$. We do so locally. Let $T\in \Sch_S$ be such that there exists $a'\in \cA'(T)$ with $\Trd_{\cA'}(a')=1$. By \Cref{canonical_f_local}, such an element describes $\underline{f'}|_T$. Since $\varphi$ is an isomorphism, $\Trd_{\cA}(\varphi^{-1}(a'))=1$ and since
\begin{align*}
\Trd_{\Cl(\cA',\sigma',f')}(c(a')\cdot \Cl(\varphi)(\und)) &= \Trd_{\Cl(\cA,\sigma,f)}(\Cl(\varphi)^{-1}(c(a'))\cdot \und)\\
&= \Trd_{\Cl(\cA,\sigma,f)}(c(\varphi^{-1}(a'))\cdot \und)
\end{align*}
this shows that $\underline{f'}|_T \circ \Cl(\varphi)|_T = \underline{f}_T$. Since there exists an open cover over which this holds, we conclude that $\underline{f'}\circ \Cl(\varphi) = \underline{f}$ globally as desired.
\end{proof}

Since the canonical semi-trace is stable under base change, the above result extends to a homomorphism of group sheaves.
\begin{lem}\label{canonical_group_sheaf_morphism}
Let $(\cA,\sigma,f)$ be a quadratic triple such that we have the canonical Clifford algebra with semi-trace $(\Cl(\cA,\sigma,f),\underline{\sigma},\underline{f})$. The morphism $\bC' \colon \PGO_{(\cA,\sigma,f)} \to \bAut(\Cl(\cA,\sigma,f),\underline{\sigma})$ factors through $\bAut(\Cl(\cA,\sigma,f),\underline{\sigma},\underline{f})$. In particular we have a commutative diagram
\begin{equation}\label{eq_action_on_canonical}
\begin{tikzcd}
\bO_{(\cA,\sigma,f)} \arrow{d}{\pi_{\bO}} \arrow{dr}{\bC} &  \\
\PGO_{(\cA,\sigma,f)} \arrow{r}{\bC'} & \cAut(\Cl(\cA,\sigma,f),\underline{\sigma},\underline{f}).
\end{tikzcd}
\end{equation}
\end{lem}
\begin{proof}
\Cref{canonical_functorial} above shows that on global sections the map $\bC'(S)$ factors through $\bAut(\Cl(\cA,\sigma,f),\underline{\sigma},\underline{f})(S)$. Combining this with \Cref{canonical_f_stable_base_change} we see that for any $T\in \Sch_S$, the homomorphism $\bC'(T)$ factors through
\begin{align*}
\Aut(\Cl(\cA|_T,\sigma|_T,f|_T),\underline{\sigma|_T},\underline{f|_T}) &= \Aut((\Cl(\cA,\sigma,f),\underline{\sigma},\underline{f})|_T) \\
&= \bAut(\Cl(\cA,\sigma,f),\underline{\sigma},\underline{f})(T).
\end{align*}
This justifies the claim.
\end{proof}

We may now use the map $\bC' \colon \PGO_{2n} \to \cAut(\Cl_0(\HH(\cV),q_{2n}),\underline{\sigma_{2n}}_0,\underline{f_{2n}})$ in the split case to twist the Clifford algebra by $\PGO_{2n}$--torsors.
\begin{prop}\label{Clifford_twisting}
Let $(\HH(\cV),q_{2n})$ be as in \Cref{split_Clifford} with $n\geq 4$ such that we have the canonical Clifford algebra with semi-trace $\cC_0 = (\Cl_0(\HH(\cV),q_{2n}),\underline{\sigma_{2n}}_0,\underline{f_{2n}})$. Give this Clifford algebra a left action of $\PGO_{2n}$ via the map $\bC'$ of \eqref{eq_action_on_canonical}. For a quadratic triple $(\cA,\sigma,f)$ of degree $2n$, consider the $\PGO_{2n}$--torsor
\[
\cP=\cIsom\big((\cEnd_{\cO}(\HH(\cV)),\sigma_{2n},f_{2n}),(\cA,\sigma,f)\big).
\]
Then, we have a natural isomorphism
\[
\cP \wedge^{\PGO_{2n}} \cC_0 \iso (\Cl(\cA,\sigma,f),\underline{\sigma},\underline{f}).
\]
\end{prop}
\begin{proof}
To show that we have an isomorphism as claimed, we define a map from the presheaf $(\cP \times \cC_0)/\sim$ of \Cref{cohomology_twisting} into $(\Cl(\cA,\sigma,f),\underline{\sigma},\underline{f})$ which is bijective locally and therefore induces an isomorphism of sheaves as claimed.

The natural map we consider is
\begin{align*}
(\cP \times \cC_0)/\sim &\to (\Cl(\cA,\sigma,f),\underline{\sigma},\underline{f}) \\
(\varphi,x) &\mapsto \Cl(\varphi)(x)
\end{align*}
We check that this is well defined. Let $\psi \in \PGO_{2n}$, then
\[
(\varphi\cdot \psi,x)= (\varphi\circ \psi,x) \mapsto \Cl(\varphi \circ \psi)(x)
\]
while
\begin{align*}
(\varphi,\psi\cdot x) &= (\varphi,\bC'(\psi)(x)) = (\varphi,\Cl(\psi)(x)) \\
&\mapsto (\Cl(\varphi)\circ \Cl(\psi))(x) = \Cl(\varphi \circ \psi)(x)
\end{align*}
which agrees with above and so we have a well-defined map. To see it is bijective locally, let $T\in \Sch_S$ such that $\cP(T)\neq \O$. Fix some $\varphi_0 \in \cP(T)$ and then because the action of $\PGO_{2n}(T)$ is simply transitive, all elements of $((\cP \times \cC_0)/\sim)(T)$ may be written uniquely as $(\varphi_0,x)$ for some $x\in \cC_0(T)$. But then it is clear that we have a bijection since it is simply the map
\[
(\varphi_0,x) \mapsto \Cl(\varphi_0)(x)
\]
where $\Cl(\varphi_0)$ is an isomorphism. This finishes the proof.
\end{proof}

\subsection{When $\cA$ is of Degree $4$}\label{A_degree_4}
If $(\cA,\sigma,f)$ is a quadratic triple of degree $4$, i.e., degree $2n$ when $n=2$, then \Cref{even_invol_type} says that the canonical involution $\underline{\sigma}$ on $\Cl(\cA,\sigma,f)$ is symplectic. If in addition we assume that $2=0 \in \cO$, then $(\Cl(\cA,\sigma,f),\underline{\sigma})$ will also be an algebra with orthogonal involution and will be eligible to be equipped with a semi-trace. Therefore, we assume $2=0\in \cO$ throughout this section. However, we will show that there does not exist a canonical choice of semi-trace on $(\Cl(\cA,\sigma,f),\underline{\sigma})$, i.e., one for which the canonical morphism $\PGO_{(\cA,\sigma,f)} \to \bAut(\Cl(\cA,\sigma,f),\underline{\sigma})$ factors through the automorphisms of the Clifford algebra with semi-trace. Intuitively, this happens for the same reason \Cref{sl_into_Alt} fails to hold when $n=2$, as pointed out in \Cref{degree_4_trace_remark}, namely that $(\Cl(\HH(\cV),q_4),\underline{\sigma_4})$ does not have enough alternating elements. These are computed in the following lemma.

\begin{lem}\label{degree_4_Alt}
Let $n=2$ and consider the split hyperbolic quadratic form $(\HH(\cV),q_4)$ of \Cref{split_Clifford}. Then,
\[
\cAlt_{(\Cl_0(\HH(\cV),q_4),\underline{\sigma_4}_0)} = \{a_0 + a_1(v_1v_1^* + v_2v_2^*) \mid a_0,a_1 \in \cO\}.
\]
\end{lem}
\begin{proof}
This follows directly from computation. Since $\HH(\cV)$ has basis $\{v_1,v_2,v_2^*,v_1^*\}$, a generic element of $\Cl_0(\HH(\cV),q_4)$ is of the form
\[
x = a_0 + a_1v_1v_2 + a_2v_1v_2^* + a_3 v_1v_1^* + a_4 v_2v_2^* + a_5v_2v_1^* + a_6v_2^*v_1^* + a_7v_1v_2v_2^*v_1^*.
\]
Since we have assumed $2=0 \in \cO$, most elements of $\Cl(\HH(\cV),q_4)$ will commute. For example, $\underline{\sigma}(v_1v_2) = v_2v_1 = -v_1v_2 = v_1v_2$. We will still have some non-trivial commutation relations, in particular $\underline{\sigma}(v_iv_i^*) = v_i^*v_i = 1+ v_iv_i^*$, and therefore one can compute that
\[
\underline{\sigma}(v_1v_2v_2^*v_1^*) = v_1^*v_2^*v_2v_1 = 1+v_1v_1^* + v_2v_2^* + v_1v_2v_2^*v_1^*.
\]
Using these relations, we obtain
\[
x-\underline{\sigma}(x) = x+\underline{\sigma}(x) = (a_3+a_4+a_7)+a_7(v_1v_1^* + v_2v_2^*).
\]
Finally, a section $a\in \cAlt_{(\Cl_0(\HH(\cV),q_4),\underline{\sigma_4}_0)}$ must be locally of the form above. Because $\Cl_0(\HH(\cV),q_4)$ is a free $\cO$--module, the restriction maps will only act on the coefficients and therefore $a$ itself is of the form above also. This justifies the claim.
\end{proof}

This relatively small submodule of alternating elements leads to the next lemma.
\begin{lem}\label{no_stable_l}
Let $n=2$ and consider the split hyperbolic quadratic form $(\HH(\cV),q_4)$ of \Cref{split_Clifford}. Assume that there exists $t\in \cO(S)$ with $t^2\neq t$. Then, for any $\ell \in \Cl_0(\HH(\cV),q_4)(S)$ with $\ell+\underline{\sigma}(\ell) = 1$, the class
\[
[\ell] \in (\Cl_0(\HH(\cV),q_4)/\cAlt_{(\Cl_0(\HH(\cV),q_4),\underline{\sigma_4}_0)})(S)
\]
is not stabilized by $\PGO_4(S)$.
\end{lem}
\begin{proof}
Towards a contradiction, let $\ell \in \Cl_0(\HH(\cV),q_4)(S)$ with $\ell+\underline{\sigma}(\ell)=1$ and assume that $\ell-\Cl(\varphi)(\ell)$ is alternating for all $\varphi \in \PGO_4(S)$. Using the calculation of $x+\underline{\sigma}(x)$ from the proof of \Cref{degree_4_Alt} above, we see that if
\[
1 = \ell+\underline{\sigma}(\ell) = (a_3+a_4+a_7)+a_7(v_1v_1^* + v_2v_2^*)
\]
then $a_7 = 0$ and $a_3+a_4 = 1$. Hence,
\[
\ell = a_0 + a_1v_1v_2 + a_2v_1v_2^* + a_3 v_1v_1^* + (1+a_3)v_2v_2^* + a_5v_2v_1^* + a_6v_2^*v_1^*.
\]
Identify $\cEnd_{\cO}(\HH(\cV)) \cong \Mat_4(\cO)$ using the ordered basis $\{v_1,v_2,v_2^*,v_1^*\}$ of $\HH(\cV)$ and consider the inner automorphism $\varphi = \Inn(B)$ with
\[
B = \begin{bmatrix} 1 & t & & \\ & 1 & & \\ & & 1 & t \\ & & & 1 \end{bmatrix} \in \bO_4(S)
\]
where $t\in \cO(S)$. For $v,w \in \HH(\cV)$, the induced automorphism on the Clifford algebra will act as $\Cl(\varphi)(vw)=(Bv)(Bw)$ due to \Cref{action_on_Clifford}, and therefore
\begin{align*}
\Cl(\varphi)(\ell) =& a_0 + a_1v_1(tv_1+v_2) + a_2v_1v_2^* + a_3v_1(tv_2^*+v_1^*) + (1+a_3)(tv_1+v_2)v_2^*\\
&+ a_5(tv_1+v_2)(tv_2^*+v_1^*) + a_6v_2^*(tv_2^*+v_1^*) \\
=& a_0 + a_1v_1v_2 + (a_2+t+t^2a_5)v_1v_2^* + (a_3+ta_5)v_1v_1^* + (1+a_3+ta_5)v_2v_2^*\\
&+ a_5v_2v_1^* + a_6v_2^*v_1^*.
\end{align*}
Thus, we obtain that
\[
\ell-\Cl(\varphi)(\ell) = (t+t^2a_5)v_1v_2^* + ta_5(v_1v_1^* + v_2v_2^*).
\]
According to \Cref{degree_4_Alt}, this element is alternating if and only if $t+t^2a_5 = 0$ for all $t\in \cO$. Choosing $t=1$ yields that $a_5=1$, and therefore $t=t^2$ for all $t \in \cO(S)$. This contradicts our assumption and hence the proof is complete.
\end{proof}

\begin{thm}\label{degree_4_no_canonical}
Let $(\cA,\sigma,f)$ be a quadratic triple of degree $4$. There does not exist a canonical semi-trace on the Clifford algebra $(\Cl(\cA,\sigma,f),\underline{\sigma})$. Precisely, we mean that for any semi-trace
\[
f' \colon \cSym_{(\Cl(\cA,\sigma,f),\underline{\sigma})} \to \cO,
\]
the group homomorphism $\bC' \colon \PGO_{(\cA,\sigma,f)} \to \bAut(\Cl(\cA,\sigma,f),\underline{\sigma})$ of \eqref{eq_twisted_Clifford_actions} does not factor through $\bAut(\Cl(\cA,\sigma,f),\underline{\sigma},f')$.
\end{thm}
\begin{proof}
Towards a contradiction, assume there does exist a semi-trace $f'$ for which $\bC'$ factors through $\bAut(\Cl(\cA,\sigma,f),\underline{\sigma},f')$. In particular, this means that for all $T\in \Sch_S$ we have a group sheaf morphism
\[
\bC'|_T \colon  \PGO_{(\cA,\sigma,f)}|_T \to \bAut(\Cl(\cA,\sigma,f),\underline{\sigma},f')|_T.
\]
Choose an affine $U \in \Sch_S$ such that $(\cA,\sigma,f)|_U \cong (\Mat_4(\cO|_U),q_4|_U,f_4|_U)$ is split. If $t^2=t$ for all $t\in \cO(U)$, then we instead restrict to the affine scheme
\[
U' = \Spec(\cO(U)[x]/\langle x^2 -1 \rangle)
\]
which has $x\in \cO(U')$ with $x^2 \neq x$. Hence, we may assume $\cO(U)$ contains a section $t$ for which $t^2 \neq t$. However, since $U$ splits $(\cA,\sigma,f)$, the morphism $\bC'|_U$ is of the form
\[
\bC'|_U \colon \PGO_4|_U \to \bAut(\Cl_0(\HH(\cV),q_4)|_U,\underline{\sigma}|_U,f'|_U).
\]
Since $U$ is affine, $f'|_U=\Trd_{\Cl_0(\HH(\cV),q_4)|_U}(\ell' \und)$ for some $\ell' \in \Cl_0(\HH(\cV),q_4)(U)$ with $\ell' + \underline{\sigma}(\ell') = 1$ by \cite[4.12]{GNR}. The fact that $\bC'|_U$ factors through automorphisms fixing $f'|_U$ implies that $\PGO_4(U)$ stabilizes the class
\[
[\ell'] \in (\Cl_0(\HH(\cV),q_4)/\cAlt_{(\Cl_0(\HH(\cV),q_4),\underline{\sigma})})(U).
\]
However, by applying \Cref{no_stable_l} with $S=U$, this is not possible. Hence no such $f'$ exists. 
\end{proof}

\end{document}